\documentclass[1 [leqno,11pt]{amsart}
\usepackage{amssymb, amsmath, color}

\let\lam=\lambda
\let\Lam=\Lambda
\let\f=\frac
\let\p=\partial
\let\D=\Delta
\let\wt=\widetilde
\let\wh=\widehat
\def\al{\alpha}
\def\e{\epsilon}

\def\du{\delta_u}

\def\h{{\rm h}}
\def\v{{\rm v}}
\def\dive{{\mathop{\rm div}\nolimits}\,}
\def\diveh{{\mathop{\rm div}_{\rm h}\nolimits}\,}

\def\nablah{\nabla_{\rm h}}
\def\uh{u^{\rm h}}
\def\dH{\dot{H}}

\def\pa{\partial}
\def\Supp{\mathop{\rm Supp}\nolimits\ }

\def\cA{{\mathcal A}}
\def\cB{{\mathcal B}}
\def\cC{{\mathcal C}}
\def\cD{{\mathcal D}}
\def\cE{{\mathcal E}}
\def\cF{{\mathcal F}}

\def\cH{{\mathcal H}}
\def\cM{{\mathcal M}}

\def\N{\mathop{\mathbb  N\kern 0pt}\nolimits}
\def\Q{\mathop{\mathbb  Q\kern 0pt}\nolimits}
\def\R{{\mathop{\mathbb R\kern 0pt}\nolimits}}
\def\Z{\mathop{\mathbb  Z\kern 0pt}\nolimits}
\def\PP{\mathop{\mathbb P\kern 0pt}\nolimits}

\def\eqdef{\buildrel\hbox{\footnotesize def}\over =}

\newcommand{\dv}{\Delta^{\rm v}}
\newcommand{\Dvl}{\Delta_{\ell}^{\rm v}}
\newcommand{\dvl}{\Delta_{\ell}^{\rm v}}
\newcommand{\Dvk}{\Delta_{k}^{\rm v}}

\newcommand{\Svl}{S_{\ell-1}^{\rm v}}
\newcommand{\Svk}{S_{k-1}^{\rm v}}
\newcommand{\Skk}{S_{k+2}^{\rm v}}

\newcommand{\BB}{B^{0,\f12}}

\newcommand{\LB}{\wt{L}_t^{\infty}(B^{0,\f12})}
\newcommand{\AB}{\wt{L}_t^2(B^{0,\f12})}
\newcommand{\LTB}{\wt{L}_T^{\infty}(B^{0,\f12})}
\newcommand{\ATB}{\wt{L}_T^2(B^{0,\f12})}

\newcommand{\andf}{~\text{ and }~}
\newcommand{\with}{~\text{ with }~}

\newtheorem{defi}{Definition}[section]
\newtheorem{thm}{Theorem}[section]
\newtheorem{lem}{Lemma}[section]
\newtheorem{rmk}{Remark}[section]
\newtheorem{col}{Corollary}[section]

\begin{document}

\title[The lifespan of $3$-D anisotropic Navier-Stokes system]
{The influence of viscous coefficients on the lifespan of
$3$-D anisotropic Navier-Stokes system}

\author[T. Hao]{Tiantian Hao}
\address[T. Hao]{
Academy of Mathematics and Systems Science, Chinese Academy of Sciences, Beijing 100190, China, and School of Mathematical Sciences, University of Chinese Academy of Sciences, Beijing 100049, China.}
\email{htt@amss.ac.cn}

\author[Y. Liu]{Yanlin Liu}
\address[Y. Liu]{School of Mathematical Sciences,
Laboratory of Mathematics and Complex Systems,
MOE, Beijing Normal University, 100875 Beijing, China.}
\email{liuyanlin@bnu.edu.cn}

\date{\today}

\begin{abstract}
The anisotropic Navier-Stokes system arises in geophysical fluid dynamics,
which is derived by changing $-\nu\D$ in the classical Navier-Stokes system
to $-(\nu_1\partial^2_1+\nu_2\partial^2_2+\nu_3\partial^2_3)$.
Here $\nu_1,\,\nu_2\,,\nu_3$ are the viscous coefficients,
which can be different from each other.
This reflects that the fluid can behave differently in each direction.

The purpose of this paper is to derive some lower bound estimates
on the lifespan to such anisotropic Navier-Stokes system.
We not only investigate the case when $\nu_1,\,\nu_2\,,\nu_3$
are all positive, but also the more sophisticated cases when
one or two of them vanish. We find that in these lower bound estimates,
the weights of $\nu_1,\,\nu_2\,,\nu_3$ are not equal.
A detailed study of this problem can also help us to have
a better understanding of the nonlinear structure in
the classical Navier-Stokes system.
\end{abstract}

\maketitle

\noindent {\sl Keywords:} Navier-Stokes system,
anisotropic, lifespan, viscous coefficients,
partial dissipation

\vskip 0.2cm
\noindent {\sl AMS Subject Classification (2000):} 35Q30, 76D03

\setcounter{equation}{0}
\section{Introduction}\label{Sect1}
%%%%%%%%%%%%%%%%%%%%%%%%%%%%%%%%%%

In this paper, we consider the lifespan of solutions to the following $3$-D incompressible anistropic Navier-Stokes system (AN-S):
 \begin{equation}\label{1.1}
     \left\{
     \begin{array}{l}
     \pa_t u-(\nu_1\partial^2_1+\nu_2\partial^2_2+\nu_3\partial^2_3)u
     +u\cdot\nabla u=-\nabla P,\qquad (t,x)\in \mathbb{R}^{+}\times\mathbb{R}^{3},\\
     \dive u = 0,\\
     u|_{t=0} =u_0,
     \end{array}
     \right.
\end{equation}
where $u$ stands for velocity field and $P$
the scalar pressure of the fluid, which guarantees the
divergence-free condition of $u$,
and $\nu_1,\,\nu_2\,,\nu_3$ are the viscous coefficients.

When $\nu_1=\nu_2=\nu_3=\nu$, \eqref{1.1} reduces to the classical
Navier-Stokes system (N-S):
\begin{equation}\label{1.2}
\pa_t u-\nu\D u+u\cdot\nabla u=-\nabla P,\quad\dive u=0.
\end{equation}
This isentropic property comes from the assumption made by Stokes
that the Cauchy stress tensor can be decomposed as
$\sigma=-P{\rm Id}+2\nu\cM$, where $P$ is the static pressure, and
$\cM=\f12(\pa_iu^j+\pa_ju^i)_{3\times3}$
is the Cauchy strain tensor. In this way, we get \eqref{1.2}.

However, in the Geophysical Fluid Dynamics (see for instance
\cite{CDGG, Pedlovsky}), the Cauchy stress tensor
in general has the following form: $\sigma=-P \rm{Id}+\tau$,
where the Turbulent stress tensor $\tau=(\nu_i\pa_iu^j
+\nu_j\pa_ju^i)_{3\times3}$. Then by using $\dive u=0$,
it is not difficult to verify that
$$\dive\tau=(\nu_1\partial^2_1+\nu_2\partial^2_2+\nu_3\partial^2_3)u,$$
which leads to the system \eqref{1.1}. Here the fact that $\nu_1,\,\nu_2\,,\nu_3$
may be different, reflects that the fluid may behave differently in each direction.
That is why we call it anisotropic.
\smallskip

In the seminal paper \cite{Leray}, among many other fundamental results,
Leray proved that the strong solutions to \eqref{1.2}
with initial data $u_0\in L^p~(3<p\leq\infty)$ must exist at least to
\begin{equation}\label{lifespanLeray}
T^{\star}\geq C\nu^{\f{p+3}{p-3}}
\|u_0\|_{L^p}^{-\f{2p}{p-3}},\quad\forall\ p\in(3,\infty].
\end{equation}
However, it is still one big open problem that for general initial data,
whether or not the corresponding strong solutions always exist globally in time.

It is worth mentioning that, the powers of $\nu$ and $\|u_0\|_{L^p}$ in
\eqref{lifespanLeray} are determined by the scaling property of N-S,
which says that if $u$ is a solution of \eqref{1.2} with initial data $u_0$ on $[0,T]$,
then $u_\lam$ is also a solution of \eqref{1.2}
with initial data $u_{0,\lam}$ on $[0,T/\lam^2]$, where
\begin{equation}\label{scaling}
u_\lam(t,x)\eqdef \lambda u(\lambda^2 t,\lambda x),\andf
u_{0,\lam}(x)\eqdef\lambda u_0(\lambda x),\quad\forall\ \lam>0.
\end{equation}
As a result, it seems hopeless so far to refine these powers in \eqref{lifespanLeray}.

The purpose of this paper is to generalize the lower bound estimate
\eqref{lifespanLeray} to AN-S, including all the cases that
either $\nu_1,\,\nu_2\,,\nu_3$ are all positive, or some of them vanish.
Notice that to estimate the lifespan, we need to
use the dissipation effect caused by the viscosity
to offset the potential accumulation caused by the convection.
Thus a detailed study of this problem can also help us to have
a better understanding of the nonlinear structure in N-S.

The first result of this paper states as follows:

\begin{thm}\label{thm1}
 {\sl Consider \eqref{1.1} with
 $0<\nu_3\leq \nu_2\leq \nu_1$ and $u_0\in L^p$.
 Let $T^{\star}$ be the lifespan of the strong solution to \eqref{1.1}.
 Then $T^{\star}$ satisfies the following lower bound estimate:
\begin{equation}\label{life div}
T^{\star}\geq C\nu^{\frac{p}{p-3}}_2(\nu_1\nu_2\nu_3)^{\frac{1}{p-3}}
\|u_0\|_{L^p}^{-\frac{2p}{p-3}},\quad\text{ when }p\in(3,\infty),
\end{equation}
and
\begin{equation}\label{infty}
T^{\star}\geq C \nu_3\|u_0\|^{-2}_{L^{\infty}},\quad\text{ when }p=\infty.
\end{equation}
}\end{thm}

\begin{rmk}\begin{enumerate}
\item[(1)] It is not difficult to verify that the lower bound estimates
\eqref{life div} and \eqref{infty} given in Theorem \ref{thm1} also
keep invariant under the scaling transformation \eqref{scaling}. As a result,
the power of $\|u_0\|_{L^p}$ and the total power of all the $\nu_i$'s
are the same as the corresponding powers in Leray's result \eqref{lifespanLeray}.

\item[(2)] The $\nu_1,\,\nu_2\,,\nu_3$ are in different statuses in
\eqref{life div} and \eqref{infty}, and the viscous coefficient
with highest power is the smallest one $\nu_3$ in \eqref{infty},
but $\nu_2$ in \eqref{life div} for finite $p$.

Another difference is that, the power of the largest viscous coefficient $\nu_1$
in \eqref{life div} is positive, which makes
$$\lim\limits_{\nu_1\rightarrow\infty}T^\star=\infty,~\text{ when }
~u_0\in L^p~\text{ for }~p\in(3,\infty).$$
In view of this and the fact that the strong solutions
to N-S will eventually become small (see \cite{GIP}),
naturally the following corollary can be expected:
\end{enumerate}\end{rmk}

 \begin{col}\label{col1.1}
{\sl Consider \eqref{1.1} with $0<\nu_3\leq \nu_2\leq \nu_1$
and $u_0\in L^2\cap L^p$ for $p\in(3,\infty)$\footnote{Only requiring
$u_0\in L^p$ is not enough to prove global existence, since if $u$ is a
global solution, then for any $\lam>0$, $u_\lam$ given by\eqref{scaling}
is also a global solution with initial data $u_{0,\lam}$. However,
as $\|u_{0,\lam}\|_{L^p}$ is not scaling invariant, it would become large
by choosing $\lam$ to be either large or small, in another word,
we can prove global existence of strong solution for any initial data
in $L^p$, which is obviously far beyond what we can do right now.}.
Then \eqref{1.1} would admit a global strong solution, provided $\nu_1$
is so large that
\begin{equation}\label{condinu11}
\nu_1\gg\nu_2^{-p-1}\nu_3^{-5(p-3)-1}\|u_0\|_{L^2}^{4(p-3)}
\|u_0\|_{L^p}^{2p}.
\end{equation}
}\end{col}

Roughly speaking, this corollary says that the strong solution to \eqref{1.1}
will become global provided only one of the viscous coefficients is sufficiently large,
which is consistent with the results obtained in \cite{LZ2, PZ4}.
Here we give an alternative point of view.
\smallskip

Next, notice there is no $\nu_1$ in \eqref{infty},
we shall give a refined estimate for the $p=\infty$ case.

\begin{thm}\label{thm2}
{\sl Consider \eqref{1.1} with $0<\nu_3\leq \nu_2\leq \nu_1$ and
$u_0\in L^{\infty}$ satisfying
$|u_0(x)|\rightarrow0$ as $|x|\rightarrow \infty$.
Let $T^\star$ be the lifespan of the corresponding strong solution.
Then we have
$$\lim_{\nu_1\rightarrow\infty}T^\star=\infty.$$
}\end{thm}

\begin{rmk}
The requirement that $u_0$ vanishes at infinity is quite natural
for N-S, which can be deduced from (but does not imply) $u_0\in L^p$
for any finite $p$. Thus not the same as \eqref{life div},
here we can not get an explicit rate how $T^\star$ tends to infinity
as $\nu_1\rightarrow\infty$.
\end{rmk}

In the following, let us turn to the more difficult situation that some of the
$\nu_i$'s vanish. First, we consider the case when one of the
viscous coefficients vanishes, i.e. $0=\nu_3<\nu_2\leq\nu_1$:
 \begin{equation}\label{ANS}
     \left\{
     \begin{array}{l}
     \partial_tu-(\nu_1\partial^2_1+\nu_2\partial^2_2)u+u\cdot\nabla u+\nabla P=0,\qquad (t,x)\in \mathbb{R}^{+}\times\mathbb{R}^{3},\\
     \dive u = 0,\\ u|_{t=0} =u_0.
     \end{array}
     \right.
 \end{equation}
This anisotropic Navier-Stokes system is used to
simulate turbulent diffusion in the Ekman layer in geophysical fluid mechanics.

Since \eqref{ANS} only has horizontal dissipation,
it is reasonable to use functional spaces
which distinguish horizontal derivatives from the vertical one,
for instance the anisotropic Besov space $B^{0,\f12}$,
see Definition \ref{defBhalf} below.
Paicu \cite{Pa02} proved the local well-posedness of \eqref{ANS}
in $B^{0,\f12}$, and also the global well-posedness
with small initial data in $B^{0,\f12}.$  This result corresponds to the
Fujita-Kato's theorem (\cite{fujitakato}) for the classical N-S.
Later, the smallness condition for global well-posedness was weaken by
the authors in \cite{PZ1, Zhang10} to be
$$\|\uh_0\|_{\cB^{0,\f12}}\exp\bigl(C\|u^3_0\|_{\cB^{0,\f12}}^4/\nu_2^4\bigr)
\ll \nu_2.$$
We refer the readers to \cite{CDGG, CZ07, Iftimie, PZ1} for well-posedness
of \eqref{ANS} in other functional spaces.

However, there seems little work concerning the lifespan of
this anisotropic system \eqref{ANS} except \cite{LZZ2}.
And our result states as follows:

\begin{thm}\label{thm3}
{\sl Consider \eqref{ANS} with $0<\nu_2\leq \nu_1$,
$u_0$ and $\nabla u_0$ in $\BB$.
Then for any $\al\in(0,\f12)$, the maximal existence time $T^\star$
of the corresponding strong solution satisfies
\begin{equation}\label{thm3.1}
|T^\star|^\al\gtrsim \f{\cC(\al)}
{\|u_0\|_{\BB}^{1-2\al}\|\nabla u_0\|_{\BB}^{2\al}},
\end{equation}
where
\begin{equation}\label{thm3.2}
\cC(\al)\eqdef
\left\{
     \begin{split}
&\nu_1^{\f14}\nu_2^{\f34-\al},
~\text{ when }\al\in(0,\f18];\\
&\nu_1^{\f38-\al}\nu_2^{\f58},
~\text{ when }\al\in(\f18,\f14];\\
&\nu_1^{\f14-\f\al2}\nu_2^{\f34-\f\al2},
~\text{ when }\al\in(\f14,\f12).
\end{split}
     \right.
\end{equation}
}\end{thm}

\begin{rmk}
In the sense of scaling, \eqref{thm3.1} corresponds to \eqref{life div}
with $p=\f{3}{1-2\al}$, which varies from $3$ to $\infty$ when
$\al$ varies from $0$ to $\f12$. One can see this more clearly by writing
\eqref{life div} as
$$|T^{\star}|^{\f{p-3}{2p}}\gtrsim\f{\nu^{\f12}_2(\nu_1\nu_2\nu_3)^{\f{1}{2p}}}
{\|u_0\|_{L^p}}\gtrsim\f{\nu^{\f12}_2(\nu_1\nu_2\nu_3)^{\f{1}{2p}}}
{\|u_0\|_{L^3}^{\f3p}\|u_0\|_{L^\infty}^{\f{p-3}p}},\quad\text{ when }p\in(3,\infty).$$
Due to the invalidity of the product law in the end-point Besov spaces,
here we did not consider the borderline case $\al=\f12$ for simplification.
But we still can expect from the expression \eqref{thm3.2}
that when $\al=\f12$, there would be no $\nu_1$
in the estimate for $T^\star$, just the same as \eqref{infty}.
\end{rmk}

\begin{col}\label{col1.2}
{\sl Consider \eqref{ANS} with $0<\nu_2\leq \nu_1$ and $u_0\in\BB$.
The corresponding strong solution will exist globally in time provided
$$\nu_1\gg \nu_2^{-3}\|u_0\|^{4}_{\BB}.$$
Moreover, the following estimate
holds uniformly in time:
$$\|u\|_{\LB}+\nu_1^{\f12}\|\p_1 u\|_{\AB}+\nu_2^{\f12} \|\p_2 u\|_{\AB}
\leq 2\| u_0\|_{\BB},\quad\forall\ t>0.$$
}\end{col}

Next, let us focus on the case when two of the viscous coefficients vanish:
\begin{equation}\label{eq4}
     \left\{
     \begin{array}{l}
     \partial_tu+u\cdot\nabla u-\nu_3\partial^2_3u+\nabla P=0,\qquad (t,x)\in \mathbb{R}^{+}\times\mathbb{R}^{3},\\
     \dive u = 0,\\ u|_{t=0} =u_0.
     \end{array}
     \right.
\end{equation}

Our result concerning the lifespan of
this anisotropic system \eqref{eq4} states as follows:

\begin{thm}\label{thm4}
{\sl Consider \eqref{eq4} with $\nu_3>0$ and $u_0\in H^{s_1,0}$
for any $s_1>2$, where the mixed Sobolev norm
$H^{s_1,0}=L^2_\v(H^{s_1}_\h)$ is defined as follows:
$$\|a\|_{H^{s_1,0}}^2=\int_{\R^3}(1+|\xi_\h|)^{2s_1}|\wh a(\xi)|^2\,d\xi,
~\text{ where }~\xi=(\xi_\h,\xi_3)\andf\xi_\h=(\xi_1,\xi_2).$$
Then \eqref{eq4} has a unique solution
$u\in L^\infty_{T^\ast}(H^{s_1,0})$ with
$\pa_3 u\in L^2_{T^\ast}(H^{s_1,0}),$
where
\begin{equation}\label{thm4.1}
T^\ast\geq C\min \Big\{\f{\nu_3^{\f13}}{\|u_0\|_{H^{s_1,0}}^{\f43}}\,,\,
\f{\nu_3^3}{\|u_0\|_{H^{s_1,0}}^4}\Big\}.
\end{equation}
}\end{thm}

\begin{rmk}
Although both \eqref{ANS} and \eqref{eq4} are the systems with partial dissipation,
but actually \eqref{ANS} behaves close to the classical N-S with full dissipation
due to the divergence-free condition, while
\eqref{eq4} is more like the Euler equations in some aspects.

For example, it is still a big open problem that whether or not
a solution to 3-D Euler equations can form singularities at finite time.
Similarly, although \eqref{thm4.1} implies
$$\lim_{\nu_3\rightarrow\infty}T^\star
\geq\lim_{\nu_3\rightarrow\infty}T^\ast=\infty,$$
but it seems difficult to prove this solution can be global
for sufficiently large $\nu_3$, just as what we have proved
in Corollaries \ref{col1.1} and \ref{col1.2} for
\eqref{1.1} and \eqref{ANS} respectively.

On the other hand, if we simply neglect the diffusion term and view \eqref{eq4}
as Euler equations and assume initially $u_0\in H^{s_2}$
for any $s_2>\f52$, then it is well-known that the corresponding
solution exists at least to
$C\|u_0\|_{H^{s_2}}^{-1}.$
Interpolating between this and \eqref{thm4.1} gives
$$T^\star\geq C\f{\nu_3^{\al}}
{\|u_0\|_{H^{s_1,0}}^{4\al}\|u_0\|_{H^{s_2}}^{1-3\al}},
\quad\forall\ \al\in\bigl[0,\f13\bigr].$$
\end{rmk}

{\bf Organization of the paper.} In Section \ref{Sect2},
we shall introduce the functional spaces used in this paper,
and give some technical lemmas. In Section \ref{Sect3},
we shall present the proof of Theorems \ref{thm1} and \ref{thm2}.
While Section \ref{Sect4} is devoted to the proof of
Theorems \ref{thm3} and \ref{thm4}.
In the last section, we shall give a brief proof
of Corollaries \ref{col1.1} and \ref{col1.2}.
\smallskip

Let us end this section with some notations we shall use throughout this paper.
\smallskip

\noindent{\bf Notations:}
We shall always denote $C$ to be an absolute constant
which may vary from line to line.
$a\lesssim b$ means that $a\leq Cb$.
For any Hilbert space $\cH$,
we designate the $\cH$ inner product of $f$ and $g$ by $(f,g)_{\cH}$,
and we shall denote $(f,g)_{L^2(\R^3)}$ simply as $(f,g)$.

For a Banach space $B$, we shall use the shorthand $L^p_T(B)$
for $\bigl\|\|\cdot\|_B\bigr\|_{L^p(0,T;dt)}$.
$\dH^s$ (resp. $H^s$) is the standard $L^2$ based
homogeneous (resp. inhomogeneous) Sobolev spaces.
Without specially mentioning, all the norms for space variables
are taken on $\R^3$, and the subscript $\h$ (resp. $\v$) is used to denote
that the norm is taken on $\R_{x_1}\times\R_{x_2}$ (resp. $\R_{x_3}$).

\setcounter{equation}{0}
\section{Functional Spaces and Some Technical Lemmas}\label{Sect2}

For the convenience of the readers, we shall first collect
some basic facts on anisotropic Littlewood-Paley theory.
Let us  recall the following dyadic operators from \cite{BCD}:
\begin{equation}\begin{split}\label{defparaproduct}
&\Delta_ja\eqdef \cF^{-1}\bigl(\varphi(2^{-j}|\xi|)\widehat{a}\bigr),
 \quad \Delta_\ell^{\rm v}a \eqdef\cF^{-1}\bigl(\varphi(2^{-\ell}|\xi_3|)\widehat{a}\bigr),\\
&S_ja\eqdef\cF^{-1}\bigl(\chi(2^{-j}|\xi|)\widehat{a}\bigr),
\quad\ S^{\rm v}_\ell a \eqdef \cF^{-1}\bigl(\chi(2^{-\ell}|\xi_3|)\widehat{a}\bigr),
\end{split}\end{equation}
where $\xi=(\xi_1,\xi_2,\xi_3)$, $\cF a$ and
$\widehat{a}$ denote the Fourier transform of $a$,
while $\cF^{-1} a$ is the inverse Fourier transform,
$\chi(\tau)$ and $\varphi(\tau)$ are smooth functions such that
\begin{align*}
&\Supp \varphi \subset \Bigl\{\tau \in \R\,: \, \frac34 \leq
|\tau| \leq \frac83 \Bigr\}\quad\mbox{and}\quad \forall
 \tau>0\,,\ \sum_{j\in\Z}\varphi(2^{-j}\tau)=1,\\
& \Supp \chi \subset \Bigl\{\tau \in \R\,: \, |\tau| \leq
\frac43 \Bigr\}\quad\mbox{and}\quad \forall
 \tau\in\R\,,\ \chi(\tau)+ \sum_{j\geq 0}\varphi(2^{-j}\tau)=1.
\end{align*}

By using these dyadic operators, let us define the following
Besov-type spaces:

\begin{defi}\label{defBhalf}
{\sl Let us consider $u\in\mathcal{S}'(\R^3)$ satisfying
$\lim_{\ell\rightarrow-\infty}\|S_\ell^\v u\|_{L^\infty}=0$. We set
$$\|u\|_{B^{0,\f12}}\eqdef \sum_{j\in\Z}2^{\f12j}\|\D^{\v}_{j}a\|_{L^2}.$$
And the corresponding Chemin-Lerner type space (see \cite{CL}) is given by
$$\|u\|_{\wt L^p_t(B^{0,\f12})}\eqdef\sum_{j\in\Z}
2^{\f12j}\|\D^{\v}_{j}a\|_{L^p([0,t];L^2)},\quad\forall\ p\in[1,\infty].$$
}\end{defi}
\begin{rmk}
Clearly, by this definition, Minkowski's inequality implies
$$\|u\|_{L^p_t(B^{0,\f12})}\leq\|u\|_{\wt L^p_t(B^{0,\f12})},
\quad\forall\ p\in[1,\infty].$$
\end{rmk}

The following anisotropic Bernstein inequality will be
frequently used in this paper.
\begin{lem}[\cite{CZ07, Pa02}]\label{lemBern}
{\sl Let $\cB_{\rm v}$ be a ball of $\R_{x_3}$,
and $\cC_{\rm v}$ be a ring of $\R_{x_3}$.
Let $1\leq p\leq\infty$ and $1\leq q_2\leq q_1\leq \infty.$ Then there holds
\begin{align*}
\Supp\wh a\subset 2^\ell\cB_{\rm v}&\Rightarrow
\|\partial_{x_3}^\alpha a\|_{L^{p}_{\rm h}(L^{q_1}_{\rm v})}
\lesssim 2^{\ell\left(|\alpha|+(1/{q_2}-1/{q_1})\right)} \|
a\|_{L^{p}_{\rm h}(L^{q_2}_{\rm v})};\\
\Supp\wh a\subset2^\ell\cC_{\rm v}&\Rightarrow
\|a\|_{L^{p}_{\rm h}(L^{q_1}_{\rm v})} \lesssim 2^{-\ell N}
\|\partial_{x_3}^N a\|_{L^{p}_{\rm h}(L^{q_1}_{\rm v})}
\eqdef 2^{-\ell N}\sup_{|\alpha|=N}
\|\partial_{x_3}^\alpha a\|_{L^{p}_{\rm
h}(L^{q_1}_{\rm v})}.
\end{align*}
}\end{lem}

Let us also recall the following two lemmas,
which will be used in the proof of Theorem \ref{thm4}.

\begin{lem}[Product law, see \cite{BCD}]\label{product law}
{\sl Let $s_1,s_2\in (-d/2,d/2)$ with $s_1+s_2>0$. Then we have
$$\|fg\|_{\dot{H}^{s_1+s_2-\f d2}(\R^d)}
\lesssim \|f\|_{\dH^{s_1}(\R^d)}\|g\|_{\dot{H}^{s_2}(\R^d)}.$$
While for the end-point case when $s_1=\f d2$, there holds for any $\e>0$ that
$$\|fg\|_{\dot{H}^{s_2}(\R^d)}\lesssim \|f\|_{\dH^{\f{d}{2}-\e}(\R^d)}^{\f12}
\|f\|_{\dH^{\f{d}{2}+\e}(\R^d)}^{\f12}\|g\|_{\dot{H}^{s_2}(\R^d)}.$$
}\end{lem}

\begin{lem}[Commutator estimate, see \cite{FMRR}]\label{communitor}
{\sl Given $s>d/2$, there is a constant $c=c(d,s)$ such that,
for all $u,B$ with $\nabla u, B \in H^s(\R^d)$,
$$\bigl\|\Lambda^s[(u\cdot \nabla)B]-(u\cdot\nabla)(\Lambda^s B)\bigr\|_{L^2(\R^d)}
\leq C\|\nabla u\|_{H^s(\R^d)}\|B\|_{H^s(\R^d)},$$
where $\Lam^s$ denotes fractional derivative operator defined as
$\cF(\Lam^s f)(\xi)=|\xi|^s\wh{f}(\xi).$
}\end{lem}
\smallskip

The following anisotropic Sobolev-type inequality
will play an essential role in this paper.

\begin{lem}\label{lem2.4}
{\sl Let $f$ be a function defined on $\R^2$, then we have
\begin{equation}\label{L42}
\|f\|_{L^4(\R^2)}\leq C\|f\|^{\f12}_{L^2(\R^2)}\|\p_1f\|^{\f14}_{L^2(\R^2)} \|\p_2f\|^{\f14}_{L^2(\R^2)},
\end{equation}
and generally for any $2\leq p<\infty$ that
\begin{equation}\label{Lp2}
\|f\|_{L^p(\R^2)}\leq C\|f\|^{\f2p}_{L^2(\R^2)}\|\p_1f\|^{\f12-\f1p}_{L^2(\R^2)} \|\p_2f\|^{\f12-\f1p}_{L^2(\R^2)}.
\end{equation}
Similarly, for the function defined on $\R^3$, there holds
\begin{equation}\label{L43}
 \|g\|_{L^4(\R^3)}\leq C\|g\|^{\f14}_{L^2(\R^3)}\|\p_1g\|^{\f14}_{L^2(\R^3)} \|\p_2g\|^{\f14}_{L^2(\R^3)}\|\p_3g\|^{\f14}_{L^2(\R^3)}.
\end{equation}
}\end{lem}

\begin{proof}
Let us first prove \eqref{L42}.
For any integers $N_1,~N_2$, by using Lemma \ref{lemBern}, we have
 \begin{equation}\begin{split}\label{LL4}
  \|f\|_{L^4(\R^2)}&\leq \sum_{k,\ell}\|\Delta_k^1 \Delta_{\ell}^2f\|_{L^4(\R^2)}\\
 &\lesssim \Big(\sum_{k\leq N_1,\ell\leq N_2}+\sum_{k\leq N_1,\ell> N_2}
 +\sum_{k> N_1,\ell\leq N_2}+\sum_{k> N_1,\ell> N_2} \Big)
 2^{\f{k+\ell}{4}}\|\Delta_k^1 \Delta_{\ell}^2f\|_{L^2(\R^2)},
\end{split}\end{equation}
where $\Delta_k^i f\eqdef\cF^{-1}\bigl(\varphi(2^{-k}|\xi_i|)\widehat{f}\,\bigr)$.
For these terms on the right-hand side, we first get
 $$\sum_{k\leq N_1,\ell\leq N_2}
 2^{\f{k+\ell}{4}}\|\Delta_k^1 \Delta_{\ell}^2f\|_{L^2(\R^2)}
 \lesssim\sum_{k\leq N_1,\ell\leq N_2}
 2^{\f{k+\ell}{4}}\|f\|_{L^2(\R^2)}
 \lesssim 2^{\frac{N_1+N_2}{4}}\|f\|_{L^2(\R^2)}.$$
 While it follows from Lemma \ref{lemBern} that
$$\sum_{k\leq N_1,\ell> N_2}2^{\f{k+\ell}{4}}\|\Delta_k^1 \Delta_{\ell}^2f\|_{L^2(\R^2)}
\lesssim \sum_{k\leq N_1,\ell> N_2}2^{\frac{k-3\ell}{4}}
\|\Delta_k^1 \Delta_{\ell}^2\p_2f\|_{L^2(\R^2)}
\lesssim 2^{\frac{N_1-3N_2}{4}}\|\p_2f\|_{L^2(\R^2)},$$
and
$$\sum_{k> N_1,\ell\leq N_2}2^{\f{k+\ell}{4}}\|\Delta_k^1 \Delta_{\ell}^2f\|_{L^2(\R^2)}
\lesssim \sum_{k> N_1,\ell\leq N_2}2^{\frac{\ell-3k}{4}}
\|\Delta_k^1 \Delta_{\ell}^2\p_1f\|_{L^2(\R^2)}
\lesssim 2^{\frac{N_2-3N_1}{4}}\|\p_1f\|_{L^2(\R^2)}.$$
Similarly, we can use Lemma \ref{lemBern} again to obtain
\begin{align*}
\sum_{k> N_1,\ell> N_2}\|\Delta_k^1 \Delta_{\ell}^2f\|_{L^2(\R^2)}
&\lesssim \sum_{k> N_1,\ell> N_2}2^{-\frac{k+\ell}{4}}
\|\Delta_k^1 \Delta_{\ell}^2\p_1f\|^{\f12}_{L^2(\R^2)}
\|\Delta_k^1 \Delta_{\ell}^2\p_2f\|^{\f12}_{L^2(\R^2)}\\
&\lesssim 2^{-\frac{N_1+N_2}{4}}\|\p_1f\|^{\f12}_{L^2(\R^2)}\|\p_1f\|^{\f12}_{L^2(\R^2)}.
\end{align*}

Now by substituting the above four estimates into \eqref{LL4},
and taking $N_1,~N_2$ so that
 $$ 2^{N_1} \sim \frac{\|\p_1f\|_{L^2(\R^2)}}{\|f\|_{L^2(\R^2)}},\andf 2^{N_2}
 \sim \frac{\|\p_2f\|_{L^2(\R^2)}}{\|f\|_{L^2(\R^2)}},$$
leads to the first desired estimate \eqref{L42}.

The other two estimates \eqref{Lp2}
and \eqref{L43} can be proved exactly along the same line.
\end{proof}

In the sequel, we shall frequently use Bony's decomposition
from \cite{Bony} in the vertical direction:
\begin{equation}\label{Bonyv}
fg=\sum_{k\in\Z}\Dvk f\cdot\Svk g+\sum_{k\in\Z}\Skk f\cdot\Dvk g.
\end{equation}

With the help of this decomposition, we shall prove the following Lemmas
\ref{lem2.5} and \ref{lem2.6}, which are the core of the proof of Theorem \ref{thm3}.

\begin{lem}\label{lem2.5}
{\sl Let $a,~b,~c$ be regular enough.
Then for any $\alpha\in[0,\f12]$,
there exists some $(d_{\ell})_{\ell\in\Z}$ on the unit sphere of $\ell^1$
such that for every $\ell\in\Z$, there holds
\begin{equation}\label{lem2.5h}
\Bigl|\int_0^t\bigl(\Dvl(a\otimes b)\,,\,\Dvl c\bigr)\,dt'\Bigr|
\lesssim t^{\alpha}2^{-\ell}d^2_{\ell}\|a\|^{2\alpha}_{\LB}
\|a\|^{1-2\alpha}_{\AB}\|b\|_{X(t)}\|c\|_{X(t)},
\end{equation}
where $\|f\|_{X(t)}\eqdef\|f\|^{\f12}_{\LB}
\|\p_1f\|^{\f14}_{\AB}\|\p_2f\|^{\f14}_{\AB}$.
If in addition $\dive a =0$, then we have
 \begin{equation}\begin{split}\label{lem2.5v}
     \Bigl|\int_0^t\bigl(\Dvl(a^{3}\cdot \p_{3}b)\,,\,\Dvl b\bigr)\,dt'\Bigr|
     &\lesssim t^{\alpha}2^{-\ell}d^2_{\ell}\|\diveh a^\h\|^{2\alpha}_{\LB}
     \|\diveh a^\h\|^{1-2\alpha}_{\AB}\|b\|_{X(t)}^2.
     \end{split}\end{equation}
}\end{lem}

\begin{proof}
We first get, by using Bony's decomposition \eqref{Bonyv} that
\begin{align*}
\bigl|\bigl(\Dvl(a\otimes b),\Dvl c\bigr)\bigr|
=\Bigl|\sum_{|k-l|\leq5}\bigl(\Dvl(\Dvk a\otimes\Svk b),\Dvl c\bigr)
+\sum_{k\geq l-4}\bigl(\Dvl(\Skk a\otimes\Dvk b),\Dvl c\bigr)\Bigr|.
\end{align*}
In view of the estimate \eqref{L42}, there holds
\begin{align*}
\Bigl|\sum_{|k-l|\leq5}&\bigl(\Dvl(\Dvk a\otimes\Svk b)\,,\,\Dvl c\bigr)\Bigr|
\lesssim \sum_{|k-l|\leq5}\|\Dvk a\|_{L^2}
\|\Svk b\|_{L_{\v}^{\infty}L_{\h}^4}\|\Dvl c\|_{L_{\v}^{2}L_{\h}^4}\\
&\lesssim\sum_{|k-l|\leq5}\|\Dvk a\|_{L^2}
\|b\|^{\f12}_{L_{\v}^{\infty}L_{\h}^2}
\|\p_1b\|^{\f14}_{L_{\v}^{\infty}L_{\h}^2}
\|\p_2b\|^{\f14}_{L_{\v}^{\infty}L_{\h}^2}
\|\Dvl c\|^{\f12}_{L^{2}}
\|\Dvl\p_1 c\|^{\f14}_{L^{2}}\|\Dvl\p_2 c\|^{\f14}_{L^{2}},
\end{align*}
and
\begin{align*}
\Bigl|\sum_{k\geq l-4}&\bigl(\Dvl(\Skk a\otimes\Dvk b)\,,\,\Dvl c\bigr)\Bigr|
\lesssim \sum_{k\geq l-4}\|\Skk a\|_{L_{\v}^{\infty}L_{\h}^2}
\|\Dvk b\|_{L_{\v}^2L_{\h}^4}\|\Dvl c\|_{L_{\v}^{2}L_{\h}^4}\\
&\lesssim\sum_{k\geq l-4}\|a\|_{L_{\v}^{\infty}L_{\h}^2}
\|\Dvk b\|^{\f12}_{L^2}\|\Dvk \p_1b\|^{\f14}_{L^2}\|\Dvk \p_2b\|^{\f14}_{L^2}
\|\Dvl c\|^{\f12}_{L^{2}}\|\Dvl\p_1 c\|^{\f14}_{L^{2}}\|\Dvl\p_2 c\|^{\f14}_{L^{2}}.
\end{align*}
As a result, we can achieve
\begin{equation}\begin{split}\label{2.10}
\Bigl|\int_0^t\bigl(\Dvl& (a \otimes b)\,,\,\Dvl c\bigr)\,dt'\Bigr|
     \lesssim t^{\alpha}\sum_{|k-l|\leq5}
     \|\Dvk a\|^{2\alpha}_{L^{\infty}_t(L^2)}
     \|\Dvk a\|^{1-2\alpha}_{L^2_t(L^2)}
     \|b\|^{\f12}_{L_t^{\infty}(L_{\v}^{\infty}L_{\h}^2)}\\
     &\times\|\p_1b\|^{\f14}_{L^2_t(L_{\v}^{\infty}L_{\h}^2)}
     \|\p_2b\|^{\f14}_{L_t^{2}(L_{\v}^{\infty}L_{\h}^2)}
     \|\Dvl c\|^{\f12}_{L_t^{\infty}(L^{2})}\|\Dvl\p_1 c\|^{\f14}_{L^2_t(L^{2})}
     \|\Dvl\p_2 c\|^{\f14}_{L_t^{2}(L^{2})}\\
     &+t^{\alpha}\sum_{k\geq l-4}\|a\|^{2\alpha}_{L^{\infty}_t(L_{\v}^{\infty}L_{\h}^2)}
     \|a\|^{1-2\alpha}_{L^2_t(L_{\v}^{\infty}L_{\h}^2)}
     \|\Dvk b\|^{\f12}_{L_t^{\infty}(L^2)}\\
     &\times\|\Dvk \p_1b\|^{\f14}_{L^2_t(L^2)}\|\Dvk \p_2b\|^{\f14}_{L_t^{2}(L^2)}
     \|\Dvl c\|^{\f12}_{L_t^{\infty}(L^{2})}\|\Dvl\p_1 c\|^{\f14}_{L^2_t(L^{2})}
     \|\Dvl\p_2 c\|^{\f14}_{L_t^{2}(L^{2})}.
\end{split}\end{equation}

Notice for any function $f$ and any $p\in[1,\infty]$, there holds
\begin{equation}\label{LLL}
\|f\|_{L^p_t(L_{\v}^{\infty}L_{\h}^2)}
\lesssim \|f\|_{L^p_t(B^{0,\f12})}\lesssim \|f\|_{\wt{L}^p_t(B^{0,\f12})},\andf
\|\Dvl f\|_{L_t^p(L^{2})}\lesssim2^{-\f12\ell}d_\ell\|f\|_{\wt L^p_t(B^{0\,\f12})}.
\end{equation}
By using \eqref{LLL}, then we can deduce from \eqref{2.10} that
\begin{align*}
\int_0^t\bigl(\Dvl(a \otimes b)\,,\,\Dvl c\bigr)\,dt'
\lesssim& t^{\alpha}\|a\|^{2\alpha}_{\LB}
\|a\|^{1-2\alpha}_{\AB}\|b\|_{X(t)}\|c\|_{X(t)}\\
&\times2^{-\f12\ell}d_\ell\Bigl(\sum_{|k-l|\leq5}2^{-\f12k}d_k
+\sum_{k\geq l-4}2^{-\f12k}d_k\Bigr)\\
\lesssim& t^{\alpha}2^{-\ell}d^2_{\ell}\|a\|^{2\alpha}_{\LB}
\|a\|^{1-2\alpha}_{\AB}\|b\|_{X(t)}\|c\|_{X(t)},
\end{align*}
which is exactly the first desired estimate \eqref{lem2.5h}.
\smallskip

Next, let us turn to the proof of \eqref{lem2.5v}.
Due to the support property of the symbols of the dyadic operators,
we can write
\begin{align*}
\dvl(a^3\cdot \p_3b)
&=\sum_{|k-\ell|\leq5}\dvl(\Svk a^3\Dvk \p_3b)
+\sum_{k\geq\ell-4}\dvl(\Dvk a^3 \Skk \p_3b)\\
&=\sum_{|k-\ell|\leq5}\Bigl([\Dvl,\Svk a^3]\Dvk \p_3b
+\Svk a^3\Dvl\Dvk \p_3b\Bigr)
+\sum_{k\geq\ell-4}\dvl(\Dvk a^3 \Skk \p_3b)\\
&=\sum_{|k-\ell|\leq5}[\Dvl,\Svk a^3]\Dvk \p_3b
+\sum_{|k-\ell|\leq5}\Bigl((\Svk a^3-\Svl a^3)\Dvl\Dvk \p_3b\Bigr)\\
&\quad+\Svl a^3\Dvl\p_3b
+\sum_{k\geq\ell-4}\dvl(\Dvk a^3 \Skk \p_3b)\eqdef I_1+I_2+I_3+I_4.
\end{align*}
In what follows, we shall handle the above term by term.

For the commutator term, we first use Taylor's formula to get
\begin{align*}
I_1(x)
=\sum_{|k-\ell|\leq5}2^{\ell}\Bigl(\int_{\R}&
h(2^{\ell}(x_3-y_3))\cdot(y_3-x_3)
\int_0^1\Svk \p_3a^3(x_\h,x_3+\tau(y_3-x_3))\,d\tau\\
& \times \Dvl\Dvk \p_3b(x_\h,y_3))\,dy_3\Bigr)(x),
\end{align*}
where $h(x_3)=\mathcal{F}^{-1}(\varphi(|\xi_3|))(x_3)$ is Schwartz.
Then in view of the estimates \eqref{L42}
and \eqref{LLL}, as well as the divergence-free condition
$\dive a=0$, we deduce
 \begin{equation}\begin{split}\label{estiI1}
\Bigl|\int_0^t\bigl(I_1\,,\,\Dvl b\bigr)\,dt'\Bigr|
&\lesssim t^{\alpha} \sum_{|k-l|\leq5}
\|\Svk \p_3a^3\|^{2\alpha}_{L^{\infty}_t(L^{\infty}_\v L^2_\h)}
\|\Svk \p_3a^3\|^{1-2\alpha}_{L^2_t(L^{\infty}_\v L^2_\h)}\\
&\quad\times\|\Dvl b\|_{L^{\infty}_t(L^2)}
\|\Dvl \p_1 b\|^{\f12}_{L^{2}_t(L^2)}
\|\Dvl \p_2 b\|^{\f12}_{L^{2}_t(L^2)}\\
&\lesssim t^{\alpha}2^{-\ell}d^2_{\ell}\|\diveh a^\h\|^{2\alpha}_{\LB}
\|\diveh a^\h\|^{1-2\alpha}_{\AB}\|b\|_{X(t)}^2.
\end{split}\end{equation}

Next, in view of the support property, we have
for any $k$ with $|k-l|\leq5$ that
$$\|\Svk a^3-\Svl a^3\|_{L^{\infty}_\v L^2_\h}
\lesssim 2^{-\f12\ell}\sum_{|\alpha|\leq5}\|\dv_{\ell+\alpha}\pa_3a^3\|_{L^2}
\lesssim 2^{-\ell}d_{\ell}\|\diveh a^{\h}\|_{B^{0,\f12}},$$
which together with Lemma \ref{lemBern} and the estimates \eqref{L42}
and \eqref{LLL} gives
\begin{equation}\begin{split}
\Bigl|\int_0^t\bigl(I_2\,,\,\Dvl b\bigr)\,dt'\Bigr|
&\lesssim t^{\alpha}
\sum_{|k-l|\leq5}\|\Svk a^3-\Svl a^3\|^{2\alpha}_{L^{\infty}_t(L^{\infty}_\v L^2_\h)}
\|\Svk a^3-\Svl a^3\|^{1-2\alpha}_{L^2_t(L^{\infty}_\v L^2_\h)}\\
&\quad\times 2^{\ell}\|\Dvl b\|_{L^{\infty}_t(L^2)}
\|\Dvl \p_1 b\|^{\f12}_{L^{2}_t(L^2)}\|\Dvl \p_2 b\|^{\f12}_{L^{2}_t(L^2)}\\
&\lesssim t^{\alpha}2^{-\ell}d^2_{\ell}\|\diveh a^\h\|^{2\alpha}_{\LB}
\|\diveh a^\h\|^{1-2\alpha}_{\AB}\| b\|_{X(t)}^2.
\end{split}\end{equation}

For the third term, we can use integration by parts to obtain
$$\bigl(I_3\,,\,\Dvl b\bigr)=-\f12\int_{\R^3}\Svl \pa_3 a^3\cdot |\Dvl b|^2\,dx
=\f12\int_{\R^3}\Svl \diveh a^\h\cdot |\Dvl b|^2\,dx,$$
which leads to
\begin{equation}\begin{split}
\Bigl|\int_0^t\bigl(I_3\,,\,\Dvl b\bigr)\,dt'\Bigr|
&\lesssim t^{\alpha}\|\Svl \diveh a^\h\|^{2\alpha}_{L^{\infty}_t(L^{\infty}_{\v}L^2_{\h})}
\|\Svl \diveh a^\h\|^{1-2\alpha}_{L^2_t(L^{\infty}_{\v}L^2_{\h})}\\
&\quad \times\|\Dvl b\|_{L^{\infty}_t(L^2)}\|\Dvl\p_1 b\|^{\f12}_{L^{2}_t(L^2)}\|\Dvl\p_2 b\|^{\f12}_{L^{2}_t(L^2)}\\
&\lesssim t^{\alpha}2^{-\ell}d^2_{\ell}\|\diveh a^\h\|^{2\alpha}_{\LB}
\|\diveh a^\h\|^{1-2\alpha}_{\AB}\|b\|_{X(t)}^2.
\end{split}\end{equation}

Finally by using Lemma \ref{lemBern},
the estimates \eqref{L42} and \eqref{LLL} again, we achieve
\begin{equation}\begin{split}\label{estiI4}
\Bigl|\int_0^t\bigl(I_4\,,\,\Dvl b\bigr)\,dt'\Bigr|
&\lesssim t^{\alpha}\sum_{k\geq l-4}2^{-k}\|\Dvk\pa_3 a^3\|^{2\alpha}_{L^{\infty}_t(L^2)}
\|\Dvk\pa_3 a^3\|^{1-2\alpha}_{L^2_t(L^2)}\\
&\quad \times2^k\|\Skk b\|^{\f12}_{L^{\infty}_t(L^{\infty}_\v L^2_\h)}
\|\Skk\p_1 b\|^{\f14}_{L^2_t(L^{\infty}_\v L^2_\h)}
\|\Skk\p_2 b\|^{\f14}_{L^{\infty}_t(L^{2}_\v L^2_\h)}\\
&\quad \times\|\Dvl b\|^{\f12}_{L^{\infty}_t(L^2)}\|\Dvl \p_1 b\|^{\f14}_{L^2_t(L^2)}
\|\Dvl \p_2 b\|^{\f14}_{L^{2}_t(L^2)}\\
&\lesssim t^{\alpha}2^{-\ell}d^2_{\ell}\|\diveh a^\h\|^{2\alpha}_{\LB}
\|\diveh a^\h\|^{1-2\alpha}_{\AB}\|b\|_{X(t)}^2.
\end{split}\end{equation}

Now combining the estimates \eqref{estiI1}-\eqref{estiI4} leads to
the scond desired estimate \eqref{lem2.5v}.
\end{proof}

\begin{rmk}\label{rmk2.2}
With a small modification to the above proof, namely by changing the index
for time integral, we can obtain a series of similar estimates, such as
\begin{align*}
\Bigl|\int_0^t\bigl(\Dvl&(a\otimes b)\,,\,\Dvl c\bigr)\,dt'\Bigr|
\lesssim t^{\f{\beta+\gamma}2}2^{-\ell}d^2_{\ell}\|a\|_{\AB}
\|b\|_{\LB}^{\f12}\|\p_1b\|_{\LB}^{\beta}\|\p_1b\|_{\AB}^{\f14-\beta}\\
&\times\|\p_2b\|_{\LB}^{\gamma}\|\p_2b\|_{\AB}^{\f14-\gamma}
\|c\|_{\LB}^{\f12}\|\p_1c\|_{\AB}^{\f14}\|\p_2c\|_{\AB}^{\f14},
\quad\forall\ \beta,\gamma\in[0,\f14].
\end{align*}
\end{rmk}

\begin{lem}\label{lem2.6}
{\sl Let $a,~b,~c$ be regular enough.
Then for any $\alpha\in[\f14,\f12)$,
there exists some $(d_{\ell})_{\ell\in\Z}$ on the unit sphere of $\ell^1$
such that for every $\ell\in\Z$, there holds
\begin{equation}\begin{split}\label{lem2.6h}
\Bigl|\int_0^t\bigl(\Dvl(a\otimes b)\,,\,\Dvl c\bigr)\,dt'\Bigr|
\lesssim& t^{\alpha}2^{-\ell}d^2_{\ell}\|a\|_{\AB}\|b\|_{\LB}^{1-2\al}
\|\p_1b\|_{\LB}^{\al}\\
&\times\|\p_2b\|_{\LB}^{\al}\|c\|_{\LB}^{2\al}
\|\p_1c\|_{\AB}^{\f12-\al}\|\p_2c\|_{\AB}^{\f12-\al}.
\end{split}\end{equation}
}\end{lem}

\begin{proof}
We first get, by using Bony's decomposition \eqref{Bonyv} that
\begin{align*}
\bigl|\bigl(\Dvl(a\otimes b)\,,\,\Dvl c\bigr)\bigr|
=\Bigl|\sum_{|k-l|\leq5}\bigl(\Dvl(\Dvk a\otimes\Svk b),\Dvl c\bigr)
+\sum_{k\geq l-4}\bigl(\Dvl(\Skk a\otimes\Dvk b),\Dvl c\bigr)\Bigr|.
\end{align*}
For $\alpha\in[\f14,\f12)$, let us take $p=\f2{1-2\al}\in[4,\infty)$.
Then by using \eqref{Lp2} with this $p$, we get
\begin{align*}
\Bigl|\sum_{|k-l|\leq5}&\bigl(\Dvl(\Dvk a\otimes\Svk b)\,,\,\Dvl c\bigr)\Bigr|
\lesssim \sum_{|k-l|\leq5}\|\Dvk a\|_{L^2}
\|\Svk b\|_{L_{\v}^{\infty}L_{\h}^p}
\|\Dvl c\|_{L_{\v}^{2}L_{\h}^{\f{2p}{p-2}}}\\
&\lesssim\sum_{|k-l|\leq5}\|\Dvk a\|_{L^2}
\|b\|^{1-2\al}_{L_{\v}^{\infty}L_{\h}^2}
\|\p_1b\|^{\al}_{L_{\v}^{\infty}L_{\h}^2}
\|\p_2b\|^{\al}_{L_{\v}^{\infty}L_{\h}^2}
\|\Dvl c\|^{2\al}_{L^{2}}
\|\Dvl\p_1 c\|^{\f12-\al}_{L^{2}}\|\Dvl\p_2 c\|^{\f12-\al}_{L^{2}},
\end{align*}
and
\begin{align*}
\Bigl|\sum_{k\geq l-4}&\bigl(\Dvl(\Skk a\otimes\Dvk b)\,,\,\Dvl c\bigr)\Bigr|
\lesssim \sum_{k\geq l-4}\|\Skk a\|_{L_{\v}^{\infty}L_{\h}^2}
\|\Dvk b\|_{L_{\v}^2L_{\h}^p}
\|\Dvl c\|_{L_{\v}^{2}L_{\h}^{\f{2p}{p-2}}}\\
&\lesssim\sum_{k\geq l-4}\|a\|_{L_{\v}^{\infty}L_{\h}^2}
\|\Dvk b\|^{1-2\al}_{L^2}\|\Dvk \p_1b\|^{\al}_{L^2}\|\Dvk \p_2b\|^{\al}_{L^2}
\|\Dvl c\|^{2\al}_{L^{2}}\|\Dvl\p_1 c\|^{\f12-\al}_{L^{2}}\|\Dvl\p_2 c\|^{\f12-\al}_{L^{2}}.
\end{align*}
As a result, we can achieve
\begin{align*}
\Bigl|\int_0^t\bigl(\Dvl(a \otimes b)\,,\,\Dvl c\bigr)\,dt'\Bigr|
\lesssim& t^{\alpha}\sum_{|k-l|\leq5}\|\Dvk a\|_{L^2_t(L^2)}
\|b\|^{1-2\al}_{L_t^{\infty}(L_{\v}^{\infty}L_{\h}^2)}
\|\p_1b\|^{\al}_{L^\infty_t(L_{\v}^{\infty}L_{\h}^2)}\\
&\quad\times\|\p_2b\|^{\al}_{L^\infty_t(L_{\v}^{\infty}L_{\h}^2)}
\|\Dvl c\|^{2\al}_{L_t^{\infty}(L^{2})}
\|\Dvl\p_1 c\|^{\f12-\al}_{L^2_t(L^{2})}
\|\Dvl\p_2 c\|^{\f12-\al}_{L^2_t(L^{2})}\\
&+t^{\alpha}\sum_{k\geq l-4}\|a\|_{L^2_t(L_{\v}^{\infty}L_{\h}^2)}
\|\Dvk b\|^{1-2\al}_{L_t^{\infty}(L^2)}
\|\Dvk \p_1b\|^{\al}_{L^\infty_t(L^2)}\\
&\quad\times\|\Dvk \p_2b\|^{\al}_{L_t^\infty(L^2)}
\|\Dvl c\|^{2\al}_{L_t^{\infty}(L^{2})}
\|\Dvl\p_1 c\|^{\f12-\al}_{L^2_t(L^{2})}
\|\Dvl\p_2 c\|^{\f12-\al}_{L_t^2(L^{2})},
\end{align*}
which together with the inequality
\eqref{LLL} leads to the desired estimate \eqref{lem2.6h}.
\end{proof}

\setcounter{equation}{0}
\section{The lifespan of AN-S with full dissipation}\label{Sect3}

In this section, we consider the case with full dissipation,
i.e. with $0<\nu_3\leq \nu_2\leq \nu_1$. The purpose of
the first subsection is to prove Theorem \ref{thm1},
while the second subsection is devoted to the
refined lower bound estimate of lifespan with $L^\infty$ initial data.

\subsection{Proof of Theorem \ref{thm1}}
According to the value of $p$, we divide the proof into 2 cases.

{\bf Case 1. when $3<p<\infty$.} In view of \eqref{1.1},
we write the equation for $u^i$ as follows:
 \begin{equation}\label{3.1}
     \pa_t u^i-(\nu_1\pa^2_1+\nu_2\pa^2_2+\nu_3\pa^2_3)u^i
     +u\cdot\nabla u^i+\pa_i P=0.
\end{equation}
Then by taking $L^2$ inner product of \eqref{3.1} with $|u|^{p-2}u^i$,
and summing up in $i$, we get
 \begin{equation}\label{3.2}
 \f{1}{p}\frac{d}{dt}\|u\|^p_{L^p}
 +\sum^3_{k=1}\nu_k\Big(\f{4(p-2)}{p^2}\bigl\|\p_k|u|^{\f{p}{2}}\bigr\|^2_{L^2}
 +\int_{\R^3}|\p_ku|^2|u|^{p-2}\,dx\Big)
 \leq\sum_{k=1}^{3}\bigl|\bigl(\p_k P , |u|^{p-2}u^k\bigr)\bigr|,
 \end{equation}
where we have used integration by parts, which makes sense since $p>3$, so that
\begin{align*}
-\sum_{i=1}^{3}\int_{\R^3}\pa_k^2 u^i\cdot|u|^{p-2}u^i\,dx
&=\int_{\R^3}\sum_{i=1}^{3}|\pa_k u^i|^2\cdot|u|^{p-2}
+\int_{\R^3}\f12\pa_k\Bigl(\sum_{i=1}^{3} |u^i|^2\Bigr)\cdot\pa_k |u|^{p-2}\,dx\\
&=\int_{\R^3}|\pa_k u|^2\cdot|u|^{p-2}
+(p-2)\pa_k |u|\cdot\pa_k |u|\cdot|u|^{p-2}\,dx,
\end{align*}
and the divergence-free condition $\dive u=0$ so that
$$\sum_{i=1}^{3}\int_{\R^3}\bigl(u\cdot\nabla u^i\bigr)\cdot|u|^{p-2}u^i\,dx
=\sum_{i=1}^{3}\int_{\R^3}\bigl(u\cdot\nabla |u^i|^2\bigr)\cdot|u|^{p-2}\,dx
=0.$$

On the other hand, we split the pressure $P$ into $P_1+P_2$ with
$$P_1=(-\Delta)^{-1}\sum_{i\neq3\text{ or }j\neq3}\p_i\p_j(u^ju^i),
\andf P_2=(-\Delta)^{-1}\p_3^2(u^3u^3).$$
Then by using integration by parts, we get
\begin{align*}
\sum_{k=1}^2\bigl|\bigl(\p_k P,|u|^{p-2}u^k\bigr)\bigr|
+\bigl|\bigl(\p_3 P_1,|u|^{p-2}u^3\bigr)\bigr|
&\lesssim\bigl\|(-\Delta)^{-1}\nabla^2(u\otimes u)\bigr\|_{L^p}
\bigl\|\nablah\bigl(|u|^{p-2}u\bigr)\bigr\|_{L^{\f{p}{p-1}}}\\
&\lesssim\|u\otimes u\|_{L^p}
\bigl\||u|^{\frac{p}{2}-1}\bigr\|_{L^{\frac{2p}{p-2}}}
\Big(\int_{\R^3}|\nabla_{\h}u|^2|u|^{p-2}\,dx\Big)^{\f12}.
\end{align*}
While Lemma \ref{lem2.4} gives
$$
\|u\otimes u\|_{L^p}
\bigl\||u|^{\f p2-1}\bigr\|_{L^{\frac{2p}{p-2}}}
=\bigl\||u|^{\f p2}\bigr\|_{L^4}^{\f4p}
\bigl\||u|^{\f p2}\bigr\|_{L^2}^{1-\f2p}
\lesssim\bigl\||u|^{\f p2}\bigr\|_{L^2}^{1-\f1p}
\prod_{k=1}^3\bigl\|\pa_k|u|^{\f p2}\bigr\|_{L^2}^{\f1p}.$$
As a result, we deduce
\begin{equation}\begin{split}\label{3.3}
\sum_{k=1}^2\bigl|\bigl(\p_k P,|u|^{p-2}u^k\bigr)\bigr|
&+\bigl|\bigl(\p_3 P_1,|u|^{p-2}u^3\bigr)\bigr|\\
&\lesssim\bigl\||u|^{\f p2}\bigr\|_{L^2}^{1-\f1p}
\prod_{k=1}^3\bigl\|\pa_k|u|^{\f p2}\bigr\|_{L^2}^{\f1p}
\Big(\int_{\R^3}|\nabla_{\h}u|^2|u|^{p-2}\,dx\Big)^{\f12}.
\end{split}\end{equation}
While for $\bigl(\p_3 P_2,|u|^{p-2}u^3\bigr)$,
we can use the divergence-free condition to get
\begin{align*}
\int_{\R^3}\p_3(-\Delta)^{-1}\p^2_3(u^3\cdot u^3)|u|^{p-2}u^3\,dx
&=2\int_{\R^3}\p^2_3(-\Delta)^{-1}\p_3u^3\cdot u^3|u|^{p-2}u^3\,dx\\
&=-2\int_{\R^3}\p^2_3(-\Delta)^{-1}\diveh\uh\cdot u^3|u|^{p-2}u^3\,dx.
\end{align*}
Then we can deduce, by using integration by parts and Lemma \ref{lem2.4} that
\begin{equation}\begin{split}\label{3.4}
|\bigl(\p_3 P_2,|u|^{p-2}u^3\bigr)|
&\lesssim\bigl\|\p^2_3(-\Delta)^{-1}(u^3)\bigr\|_{L^{2p}}
\bigl\|\nablah\bigl(u^3|u|^{p-2}u^3\bigr)\bigr\|_{L^{\f{2p}{2p-1}}}\\
&\lesssim\|u\|_{L^{2p}}
\bigl\||u|^{\f p2}\bigr\|_{L^{\frac{2p}{p-1}}}
\Big(\int_{\R^3}|\nabla_{\h}u|^2|u|^{p-2}\,dx\Big)^{\f12}\\
&\lesssim\bigl\||u|^{\f p2}\bigr\|_{L^2}^{1-\f1p}
\prod_{k=1}^3\bigl\|\pa_k|u|^{\f p2}\bigr\|_{L^2}^{\f1p}
\Big(\int_{\R^3}|\nabla_{\h}u|^2|u|^{p-2}\,dx\Big)^{\f12}.
\end{split}\end{equation}

Now by substituting the estimates \eqref{3.3} and \eqref{3.4} into
\eqref{3.2}, and then integrating the resulting inequality in time, we achieve
\begin{equation}\begin{split}\label{3.5}
\|u\|^p_{L^{\infty}_t(L^p)}&
+\sum_{k=1}^3\nu_k\Bigl(\f{4(p-2)}{p^2}\bigl\|\p_k|u|^{\f{p}{2}}\bigr\|^2_{L^2_t(L^2)}
+\int_0^t\int_{\R^3}|\p_ku|^2|u|^{p-2}\,dx\Bigr)\\
&\leq \|u_0\|^p_{L^p}+Ct^{\frac{p-3}{2p}}\|u\|^{\frac{p-1}{2}}_{L^{\infty}_t(L^p)}
\prod_{k=1}^3\bigl\|\pa_k|u|^{\f p2}\bigr\|_{L^2_t(L^2)}^{\f1p}
\Bigl(\int_0^t\int_{\R^3}|\nabla_{\h}u|^2|u|^{p-2}\,dx\Bigr)^{\f12}.
\end{split}\end{equation}

In the following, we shall use a standard continuity argument to get
a lower bound estimate for the lifespan $T^\star$ for the solution $u$.
To do this, let us denote
\begin{align*}
T^{\ast}\eqdef \sup\Big\{&T>0:u\text{ exists on $[0,T)$, and for any $t<T$, there holds}\\\
&\|u\|^p_{L^{\infty}_t(L^p)}+\sum^3_{k=1}\nu_k
\Big(\f{4(p-2)}{p^2}\bigl\|\p_k|u|^{\f{p}{2}}\bigr\|^2_{L^2_t(L^2)}
+\int_0^t\int_{\R^3}|\p_ku|^2|u|^{p-2}\,dx\Big)
\leq 2\|u_0\|^p_{L^p}\Big\},
\end{align*}
which is positive due to the local well-posedness result.
Then for $0<t< T^{\ast}$, \eqref{3.5} implies
\begin{align*}
\|u\|^p_{L^{\infty}_t(L^p)}+\sum_{k=1}^3\nu_k
\Big(\f{4(p-2)}{p^2}&\bigl\|\p_k|u|^{\f{p}{2}}\bigr\|^2_{L^2_t(L^2)}
+\int_0^t\int_{\R^3}|\p_ku|^2|u|^{p-2}\,dx\Big)\\
&\leq\|u_0\|^p_{L^p}
+C|T^{\ast}|^{\frac{p-3}{2p}}\nu_2^{-\f{1}{2}}
(\nu_1\nu_2\nu_3)^{-\frac{1}{2p}}\|u_0\|^{p+1}_{L^p}.
\end{align*}
Then in view of the definition of $T^\ast$, there must hold
$$C|T^{\ast}|^{\frac{p-3}{2p}}\nu_2^{-\f{1}{2}}
(\nu_1\nu_2\nu_3)^{-\frac{1}{2p}}\|u_0\|^{p+1}_{L^p}
>\|u_0\|^p_{L^p},$$
which guarantees the existence time of the solution $u$ is at least
$$T^\star\geq T^{\ast}> C\nu^{\frac{p}{p-3}}_2(\nu_1\nu_2\nu_3)^{\frac{1}{p-3}}\|u_0\|^{-\frac{2p}{p-3}}_{L^p}.$$
This is exactly the first desired lower bound estimate \eqref{life div} when $3<p<\infty$.
\smallskip

{\bf Case 2. when $p=\infty$.} Instead of energy estimate
and continuity argument,
the strategy for $L^\infty$ case is to use the following
Banach fixed point theorem:
\begin{lem}[Lemma $5.5$ in \cite{BCD}]\label{lemfix}
{\sl Let $E$ be a Banach space, $\cB$ a continuous bilinear map from
$E\times E$ to $E$, and $\al$ a positive number such that
$$\al<\f{1}{4\|\cB\|}\with\|\cB\|\eqdef\sup_{\|u\|,\|v\|\leq1}
\|\cB(u,v)\|.$$
Then for any $a$ in the ball $B(0,\al)$ in $E$, a unique $x$
exists in $B(0,2\al)$ in $E$ such that
$$x=a+\cB(x,x).$$
}\end{lem}
Let us continue our proof.
First, it is standard to write \eqref{1.1} in the integral form:
\begin{equation}\label{3.6}
u(t)=e^{t(\nu_1\p^{2}_1+\nu_2\p^{2}_2+\nu_3\p^{2}_3)}u_0
+\int_0^t e^{(t-s)(\nu_1\p^{2}_1+\nu_2\p^{2}_2+\nu_3\p^{2}_3)}\PP\dive(u\otimes u)(s)\,ds,
\end{equation}
where $\PP$ is the Leray projection operator on divergence free vector fields.

For the linear part, there holds
$$\bigl\|e^{t(\nu_1\p^{2}_1+\nu_2\p^{2}_2+\nu_3\p^{2}_3)}u_0\bigr\|
_{L^{\infty}_t(L^{\infty})}\leq\|u_0\|_{L^{\infty}}.$$
While for the bilinear part, in view of the assumption
$0<\nu_3\leq \nu_2\leq \nu_1$, we have
\begin{align*}
\bigl\|\int_0^t e^{(t-s)(\nu_1\p^{2}_1+\nu_2\p^{2}_2+\nu_3\p^{2}_3)}
 \PP\dive(u\otimes v)\bigr\|_{L^{\infty}_t(L^{\infty})}
 &\lesssim \nu_3^{-\f12}\int_0^t \frac{1}{\sqrt{t-s}}\,ds
 \cdot\|u\|_{L^{\infty}_t(L^{\infty})}\|v\|_{L^{\infty}_t(L^{\infty})}\\
 &\lesssim \nu_3^{-\f12}\sqrt{t}
 \cdot\|u\|_{L^{\infty}_t(L^{\infty})}\|v\|_{L^{\infty}_t(L^{\infty})}.
\end{align*}
Then in view of Lemma \ref{lemfix},
as long as
$$\|u_0\|_{L^{\infty}}<C\nu_3^{\f12} t^{-\f12},\quad\text{i.e.}\quad
t<C\nu_3\|u_0\|_{L^{\infty}}^{-2},$$
we can seek a unique strong solution to \eqref{3.6}. This means that the lifespan
of $u$ at least satisfies
$$T^{\star}\geq C \nu_3\|u_0\|^{-2}_{L^{\infty}}.$$
Moreover, this solution remains in the ball with radius
$2\|u_0\|_{L^{\infty}}$ before $T^{\star}$, namely
$$\|u(t)\|_{L^{\infty}}\leq2\|u_0\|_{L^{\infty}},\quad\forall\ t\leq T^{\star}.$$

Combining the above two cases, now we have completed the proof of Theorem \ref{thm1}.

\subsection{Proof of Theorem \ref{thm2}}
Since $u_0\in L^\infty$ vanishes at infinity, for any $\e>0$,
we can split $u_0$ into $u_{1,0}+u_{2,0}$
with $u_{1,0}\in L^4\cap L^{\infty}$ and $\|u_{2,0}\|_{L^{\infty}}<\e$.
We define $u_2$ to be the solution of the following Navier-Stokes system:
\begin{equation}\label{3.7}
\left\{
\begin{array}{l}
\partial_tu_2+u_2\cdot\nabla u_2
-(\nu_1\pa^2_1+\nu_2\pa_2^2+\nu_3\pa_3^2) u_2+\nabla\bar P=0,\\
\dive u_2 = 0,\\
u_2|_{t=0} =u_{2,0},
\end{array}
\right.
\end{equation}
for some properly chosen pressure function $\bar P$.
And then it follows from \eqref{1.1} and \eqref{3.7} that
$u_1\eqdef u-u_2$ and $\Pi\eqdef P-\bar P$ solve the following perturbed system:
\begin{equation}\label{3.8}
     \left\{
     \begin{array}{l}
	 \partial_tu_1-(\nu_1\pa^2_1+\nu_2\pa_2^2+\nu_3\pa_3^2)u_1
+u\cdot\nabla u_1+u_1\cdot\nabla u_2+\nabla\Pi=0,\\
     \dive u_1 = 0,\\
      u_1|_{t=0} =u_{1,0}.
     \end{array}
     \right.
\end{equation}

We first get, by using Theorem \ref{thm1} and the restriction
$\|u_{2,0}\|_{L^{\infty}}<\e$ that
the solution $u_2$ to \eqref{3.7} at least exists on
$[0,T_2]$, where $T_2\eqdef C\nu_3\e^{-2}$. Moreover,
there holds
\begin{equation}\label{boundu2}
\|u_2(t)\|_{L^{\infty}}<2\e,\quad\forall\ t\in[0,T_2].
\end{equation}

Next, by taking $L^2$ inner product
of the first equation in \eqref{3.8} with $|u_1|^2 u_1,$ we get
 \begin{equation}\label{3.10}
 \f{1}{4}\frac{d}{dt}\|u_1\|^4_{L^4}+\sum_{k=1}^3\nu_k
 \Big(\int_{\R^3}|\p_ku_1|^2|u_1|^2\,dx
 +\f12\bigl\|\p_k|u_1|^{2}\bigr\|^2_{L^2}\Big)
 \leq \bigl|\bigl(u_1\cdot \nabla u_2+\nabla\Pi,|u_1|^{2}u_1\bigr)\bigr|.
 \end{equation}
For the terms on the right-hand side,
we first get, by using integration by parts that
$$\bigl|\bigl(u\cdot\nabla u^i_2,|u_1|^{2}u^i_1\bigr)\bigr|
=\Bigl|\sum_{j=1}^3\int_{\R^3}u^i_2u_1^j\p_ju_1^i|u_1|^2
+u^i_2u_1^ju_1^i\p_j|u_1|^2\,dx\Bigr|.$$
Then by using H\"older inequality and Young's inequality, we get
\begin{equation}\begin{split}\label{3.11}
\bigl|\bigl(u_1\cdot \nabla u_2,|u_1|^{2}u_1\bigr)\bigr|
&\leq \|u_2\|_{L^{\infty}}\|u_1\|^2_{L^4}
\Bigl(\big(\int_{\R^3}|\nabla u_1|^2|u_1|^2\,dx\big)^{\f12}
+\bigl\|\nabla|u_1|^{2}\bigr\|_{L^2}\Bigr)\\
&\leq \f{\nu_3}{8}\Bigl(\int_{\R^3}|\nabla u_1|^2|u_1|^2\,dx
+\bigl\|\nabla|u_1|^{2}\bigr\|^2_{L^2}\Bigr)
+\f{C}{\nu_3}\|u_2\|^2_{L^{\infty}}\|u_1\|^4_{L^4}.
\end{split}\end{equation}

As for the pressure term $\Pi$, we write it as
\begin{align*}
\nabla\Pi&=\nabla(-\Delta)^{-1}\sum_{i,j=1}^3
\pa_i\pa_j(u_1^iu_1^j+2u_2^iu_1^j)\\
&=(-\Delta)^{-1}\sum_{i,j=1}^3\nabla
\pa_j(u_1\cdot\nabla u_1^j+2 u_2\cdot\nabla u_1^j)\eqdef II_1+II_2.
\end{align*}
Notice that $(-\Delta)^{-1}\nabla\pa_j$ is a Fourier multiplier
of degree $0$, then we can get
\begin{equation}\begin{split}\label{3.12}
\bigl|\bigl(II_1,|u_1|^{2}u_1\bigr)\bigr|
&\leq C\|u_1\|_{L^4}\bigl\||u_1|^2\bigr\|_{L^{4}}
\big(\int_{\R^3}|\nabla u_1|^2|u_1|^2\,dx\big)^{\f12}\\
&\leq \|u_1\|^{\f32}_{L^4}\bigl\|\p_1|u_1|^2\bigr\|^{\f14}_{L^{2}}
\bigl\|\p_2|u_1|^2\bigr\|^{\f14}_{L^{2}}\bigl\|\p_3|u_1|^2\bigr\|^{\f14}_{L^{2}}
\big(\int_{\R^3}|\nabla u_1|^2|u_1|^2\,dx\big)^{\f12}\\
&\leq\f{\nu_3}8\int_{\R^3}|\nabla u_1|^2|u_1|^2\,dx
+\f{C}{\nu_3}\|u_1\|^{3}_{L^4}\bigl\|\p_1|u_1|^2\bigr\|^{\f12}_{L^{2}}
\bigl\|\p_2|u_1|^2\bigr\|^{\f12}_{L^{2}}\bigl\|\p_3|u_1|^2\bigr\|^{\f12}_{L^{2}},
\end{split}\end{equation}
where in the second step, we have used \eqref{L43} again.
Similarly, there holds
\begin{equation}\begin{split}\label{3.13}
\bigl|\bigl(II_2,|u_1|^{2}u_1\bigr)\bigr|
&\leq C\|u_2\|_{L^\infty}\|u_1\|_{L^{4}}^2
\big(\int_{\R^3}|\nabla u_1|^2|u_1|^2\,dx\big)^{\f12}\\
&\leq\f{\nu_3}8\int_{\R^3}|\nabla u_1|^2|u_1|^2\,dx
+\f{C}{\nu_3}\|u_2\|^2_{L^{\infty}}\|u_1\|^4_{L^4}.
\end{split}\end{equation}

Now by submiting \eqref{3.11}-\eqref{3.13} into \eqref{3.10}
and then integrating in time, we achieve
\begin{equation}\begin{split}\label{3.14}
&\|u_1\|^4_{L^{\infty}_t(L^4)}+\sum_{k=1}^3\nu_k
\Big(\bigl\||\p_ku_1|\cdot|u_1|\bigr\|_{L^2_t(L^2)}^2
+\bigl\|\p_k|u_1|^{2}\bigr\|^2_{L^2_t(L^2)}\Big)\\
&\leq \|u_{1,0}\|^4_{L^4}+\f{C}{\nu_3}t\|u_2\|^2_{L^{\infty}_t(L^{\infty})}
\|u_1\|^4_{L^{\infty}_t(L^4)}\\
&\qquad+\f{C}{\nu_3}t^{\f14}\|u_1\|^{3}_{L^{\infty}_t(L^4)}
\bigl\|\p_1|u_1|^{2}\bigr\|^{\f12}_{L^2_t(L^2)}
\bigl\|\p_2|u_1|^{2}\bigr\|^{\f12}_{L^2_t(L^2)}
\bigl\|\p_3|u_1|^{2}\bigr\|^{\f12}_{L^2_t(L^2)}.
\end{split}\end{equation}
 Let us denote
 \begin{align*}
   T^{\ast}_1\eqdef \sup \Big\{&0<T\leq T_2:
   u\text{ exists on $[0,T)$, and for any $t<T$, there holds}\\
   &\|u_1\|^4_{L^{\infty}_t(L^4)}
   +\sum_{k=1}^3\nu_k\Big(\bigl\||\p_ku_1|\cdot|u_1|\bigr\|_{L^2_t(L^2)}^2
   +\bigl\|\p_k|u_1|^{2}\bigr\|^2_{L^2_t(L^2)}\Big)<2\|u_{1,0}\|^4_{L^4}\Big\}.
 \end{align*}
Then for any $t<T^{\ast}_1$, we can obtain from
\eqref{boundu2} and \eqref{3.14} that
\begin{equation}\begin{split}\label{3.15}
\|u_1\|^4_{L^{\infty}_t(L^4)}
+&\sum_{k=1}^3\nu_k\Big(\bigl\||\p_ku_1|\cdot|u_1|\bigr\|_{L^2_t(L^2)}^2
+\bigl\|\p_k|u_1|^{2}\bigr\|^2_{L^2_t(L^2)}\Big)\\
&\leq \|u_{1,0}\|^4_{L^4}+\f{C}{\nu_3}
\Bigl(t\e^2\|u_{1,0}\|^4_{L^4}
+t^{\f14}\nu_1^{-\f14}\nu_2^{-\f14}\nu_3^{-\f14}\|u_{1,0}\|^6_{L^4}\Bigr).
\end{split}\end{equation}
Thus if
$$T^{\ast}_1
<\min\Bigl\{\f{\nu_3}{4C}\e^{-2}\,,\,
\bigl(\f{\nu_3}{4C}\bigr)^4\nu_1\nu_2\nu_3\|u_{1,0}\|_{L^4}^{-8}\Bigr\},$$
then we can deduce from \eqref{3.15} that for any $t<T^{\star}_1$, there holds
$$\|u_1\|^4_{L^{\infty}_t(L^4)}
+\sum_{k=1}^3\nu_k\Big(\bigl\||\p_ku_1|\cdot|u_1|\bigr\|_{L^2_t(L^2)}^2
+\bigl\|\p_k|u_1|^{2}\bigr\|^2_{L^2_t(L^2)}\Big)
\leq \f32\|u_{1,0}\|^4_{L^4},$$
which contradict to the definition of $T^{\star}_1$.
As a result, actually there must hold
\begin{equation}\label{3.16}
T^{\ast}_1\geq\min\Bigl\{T_2\,,\,\f{\nu_3}{4C}\e^{-2}\,,\,
\bigl(\f{\nu_3}{4C}\bigr)^4\nu_1\nu_2\nu_3\|u_{1,0}\|_{L^4}^{-8}\Bigr\}.
\end{equation}
Hence for fixed $u_0$ and sufficiently large $\nu_1$, we can take
$\e$ to be $\nu_1^{-1}$ which is sufficiently small.
And then it is not difficult to derive from the chosen of
$T_2= C\nu_3\e^{-2}$ and \eqref{3.16} that
$$\lim_{\nu_1\rightarrow\infty}T_2=\lim_{\nu_1\rightarrow\infty}T_1^\ast=\infty,$$
which in particular implies the lifespan of $u=u_1+u_2$
also tends to infinity as $\nu_1\rightarrow\infty$.
This completes the proof of Theorem \ref{thm2}.

\setcounter{equation}{0}
\section{The lifespan of AN-S with partial dissipation}\label{Sect4}

\subsection{Proof of Theorem \ref{thm3}}\label{sub4.1}
{\bf Step 1. $\BB$ estimate of $u$.}
By applying $\Dvl$ to \eqref{ANS},
and then taking $L^2$ inner product of the resulting equation with $\Dvl u$ gives
$$\f12\f{d}{dt}\|\Dvl u\|^2_{L^2}
+\nu_1 \|\Dvl \p_1 u\|^2_{L^2}+\nu_2 \|\Dvl \p_2 u\|^2_{L^2}
\leq\bigl|\bigl(\Dvl(u^{\h}\cdot \nabla_{\h}u)\,,\,\Dvl u \bigr)
+\bigl(\Dvl(u^3\cdot \p_3u)\,,\,\Dvl u \bigr)\bigr|.$$
Integrating in time and then using Lemma \ref{lem2.5},
we get for any $\alpha\in(0,\f12)$, there holds
\begin{equation}\begin{split}\label{4.1}
\|\Dvl u^{\h}\|^2_{L^{\infty}_t(L^2)}&+2\nu_1 \|\Dvl \p_1 u^{\h}\|^2_{L^2_t(L^2)}
+2\nu_2 \|\Dvl \p_2 u^{\h}\|^2_{L^2_t(L^2)}\\
&\leq \|\Dvl u_0^{\h}\|^2_{L^2}
+C t^{\alpha}2^{-\ell}d^2_{\ell}\|\nabla_{\h}u\|^{2\alpha}_{\LB}
\|\nabla_{\h}u\|^{1-2\alpha}_{\AB}\|u\|_{X(t)}^2,
\end{split}\end{equation}
where $\|f\|_{X(t)}\eqdef\|f\|^{\f12}_{\LB}
\|\p_1f\|^{\f14}_{\AB}\|\p_2f\|^{\f14}_{\AB}$.
Taking square root of \eqref{4.1}, multiplying $2^\ell$ to both sides
of the resulting inequality, and then summing up in $\ell$ gives
\begin{equation}\begin{split}\label{4.2}
\|u\|_{\LB}&+\nu_1^{\f12} \|\p_1 u\|_{\AB}+\nu_2^{\f12}\|\p_2 u\|_{\AB}\\
&\leq \|u_0\|_{\BB}+Ct^{\frac{\alpha}{2}}\|u\|_{X(t)}
\|\nabla_{\h}u\|^{\alpha}_{\LB}\|\nabla_{\h}u\|^{\f12-\alpha}_{\AB}.
\end{split}\end{equation}

{\bf Step 2. $\BB$ estimate of $\nabla u$.}
We first write the equation satisfied by $\dvl\nabla u$ as
\begin{equation}\label{4.3}
\pa_t\dvl\nabla u-(\nu_1\pa^2_1+\nu_2\pa^2_2)\dvl\nabla u
+\dvl\nabla(u\cdot\nabla u)+\nabla\dvl\nabla P=0.
\end{equation}
By taking $L^2$ inner product of \eqref{4.3} with $\dvl\nabla u$, we obtain
\begin{equation}\begin{split}\label{4.4}
&\f12\|\Dvl\nabla u\|^2_{L^{\infty}_t(L^2)}
+\sum_{i=1}^2\nu_i\|\pa_i\Dvl\nabla u\|^2_{L^2_t(L^2)}
\leq \f12\|\Dvl\nabla u_0\|^2_{L^2}+\cA+\cB,~\text{ where}\\
\cA&\eqdef\int_0^t\Bigl(\Dvl\nabla(\uh\cdot \nablah u)
\,,\,\Dvl \nabla u \Bigr)\,dt',\quad
\cB\eqdef\int_0^t\Bigl(\Dvl\nabla(u^{3}\cdot \p_3u)
\,,\,\Dvl\nabla u \Bigr)\,dt'.
\end{split}\end{equation}
And in view of Leibniz formula, we can write
$$\cA=\int_0^t\Bigl(\Dvl(\nabla u^{\h}\cdot \nabla_{\h}u)\,,\,\Dvl \nabla u \Bigr)\,dt'
+\int_0^t\Bigl(\Dvl(\uh\cdot \nablah\nabla u)\,,\,\Dvl \nabla u \Bigr)\,dt'
\eqdef \cA_1+\cA_2,$$
and
$$\cB=\int_0^t\Bigl(\Dvl(\nabla u^3\cdot \pa_3u)\,,\,\Dvl \nabla u \Bigr)\,dt'
+\int_0^t\Bigl(\Dvl(u^3\cdot \pa_3\nabla u)\,,\,\Dvl \nabla u \Bigr)\,dt'
\eqdef \cB_1+\cB_2.$$

Now let us handle these term by term.
By using \eqref{lem2.5h} with $a=\nablah u,~b=\nabla\uh$
and $c=\nabla u$, we can obtain for any $\al\in(0,\f12)$ that
\begin{equation}\label{4.5}
     \begin{aligned}
     |\cA_1| \lesssim& t^{\alpha} 2^{-\ell}d^2_{\ell}
     \|\nabla_{\h}u\|^{2\alpha}_{\LB}
     \|\nabla_{\h}u\|^{1-2\alpha}_{\AB}\\
&\times\|\nabla u\|_{\LB}
\|\p_1\nablah u\|^{\f12}_{\AB}\|\p_2\nablah u\|^{\f12}_{\AB}.
     \end{aligned}
 \end{equation}
While in view of Remark \ref{rmk2.2}, we get for $\alpha\in(0,\f18]$ that
 \begin{equation}
     \begin{aligned}
|\cA_2| \lesssim t^{\alpha}2^{-\ell}d^2_{\ell}
&\|\uh\|^{\f12}_{\LB}\|\p_1\uh\|^{\f14}_{\AB}
\|\p_2\uh\|^{2\alpha}_{\LB}\|\p_2\uh\|^{\f14-2\alpha}_{\AB}\\
&\times \|\nablah\nabla u\|_{\AB}\|\nabla u\|^{\f12}_{\LB}
\|\p_1\nabla u\|^{\f14}_{\AB}\|\p_2\nabla u\|^{\f14}_{\AB},
     \end{aligned}
 \end{equation}
and for $\alpha\in[\f18,\f14]$ that
 \begin{equation}
     \begin{aligned}
|\cA_2| \lesssim t^{\al}2^{-\ell}d^2_{\ell}
&\|\uh\|^{\f12}_{\LB}\|\p_1\uh\|^{2\alpha-\f14}_{\LB}
\|\p_1\uh\|^{\f12-2\alpha}_{\AB}\|\p_2\uh\|^{\f14}_{\LB}\\
&\times \|\nablah\nabla u\|_{\AB}\|\nabla u\|^{\f12}_{\LB}
\|\p_1\nabla u\|^{\f14}_{\AB}\|\p_2\nabla u\|^{\f14}_{\AB}.
     \end{aligned}
 \end{equation}
As for the case when $\al\in[\f14,\f12)$, we can use \eqref{lem2.6h} to deduce
\begin{equation}\begin{split}
|\cA_2|
\lesssim& t^{\alpha}2^{-\ell}d^2_{\ell}\|\uh\|_{\LB}^{1-2\al}
\|\p_1\uh\|_{\LB}^{\al}\|\p_2\uh\|_{\LB}^{\al}\\
&\times\|\nablah\nabla u\|_{\AB}\|\nabla u\|_{\LB}^{2\al}
\|\p_1\nabla u\|_{\AB}^{\f12-\al}\|\p_2\nabla u\|_{\AB}^{\f12-\al}.
\end{split}\end{equation}

On the other hand, thanks to \eqref{lem2.5h},
there holds for any $\al\in(0,\f12)$ that
\begin{align*}
|\cB_1|\lesssim& t^{\alpha}2^{-\ell}d^2_{\ell}
\|\nabla u^3\|^{2\alpha}_{\LB}\|\nabla u^3\|^{1-2\alpha}_{\AB}
\|\p_3u\|_{\LB}^{\f12}
\|\p_1\pa_3 u\|^{\f14}_{\AB}\\
&\times\|\p_2\pa_3 u\|^{\f14}_{\AB}
\|\nabla u\|_{\LB}^{\f12}
\|\p_1\nabla u\|^{\f14}_{\AB}\|\p_2\nabla u\|^{\f14}_{\AB}.
\end{align*}
In addition, the divergence-free condition on $u$ tells us that
$$\nabla u^3=(\nablah u^3,\pa_3 u^3)
=(\nablah u^3,-\diveh\uh).$$
As a result, we can obtain
\begin{equation}\begin{split}
|\cB_1|\lesssim& t^{\alpha}2^{-\ell}d^2_{\ell}
\|\nablah u\|^{2\alpha}_{\LB}\|\nablah u\|^{1-2\alpha}_{\AB}
\|\p_3u\|_{\LB}^{\f12}
\|\p_1\pa_3 u\|^{\f14}_{\AB}\\
&\times\|\p_2\pa_3 u\|^{\f14}_{\AB}
\|\nabla u\|_{\LB}^{\f12}
\|\p_1\nabla u\|^{\f14}_{\AB}\|\p_2\nabla u\|^{\f14}_{\AB}.
\end{split}\end{equation}
While by using the communitor estimate \eqref{lem2.5v},
there holds for any $\al\in(0,\f12)$ that
\begin{equation}\begin{split}\label{4.10}
|\cB_2| \lesssim& t^{\alpha}2^{-\ell}d^2_{\ell}
\|\diveh\uh\|^{2\alpha}_{\LB}\|\diveh\uh\|^{1-2\alpha}_{\AB}\\
&\times \|\nabla u\|_{\LB}\|\p_1\nabla u\|^{\f12}_{\AB}
\|\p_2\nabla u\|^{\f12}_{\AB}.
\end{split}\end{equation}

Now, by substituting \eqref{4.5}-\eqref{4.10} into \eqref{4.4},
we achieve for any $\alpha\in(0,\f12)$ that
\begin{equation}\begin{split}\label{4.11}
\|\nabla u&\|_{\LB}+\sum_{i=1}^2\nu_i^{\f12}\|\p_i\nabla u\|_{\AB}
\leq \|\nabla u_0\|_{\BB}+Ct^{\f\al2}\Bigl(F(\al)\\
&+\|\nablah u\|^{\alpha}_{\LB}\|\nablah u\|^{\f12-\alpha}_{\AB}
\|\nabla u\|^{\f12}_{\LB}
\|\p_1\nabla u\|^{\f14}_{\AB}\|\p_2\nabla u\|^{\f14}_{\AB}\Bigr),
\end{split}\end{equation}
where the function $F(\al)$ is taken as follows:
\begin{equation*}
F(\al)=
\left\{
     \begin{split}
&\|u\|^{\f14}_{\LB}\|\p_1u\|^{\f18}_{\AB}
\|\p_2u\|^{\alpha}_{\LB}\|\p_2u\|^{\f18-\alpha}_{\AB}
\|\nablah\nabla u\|^{\f12}_{\AB}\\
&\quad\times\|\nabla u\|^{\f14}_{\LB}
\|\p_1\nabla u\|^{\f18}_{\AB}\|\p_2\nabla u\|^{\f18}_{\AB},
~\text{ when }\al\in(0,\f18];\\
&\|u\|^{\f14}_{\LB}\|\p_1u\|^{\al-\f18}_{\LB}\|\p_1u\|^{\f14-\al}_{\AB}
\|\p_2u\|^{\f18}_{\LB}\|\nablah\nabla u\|^{\f12}_{\AB}\\
&\quad \times \|\nabla u\|^{\f14}_{\LB}
\|\p_1\nabla u\|^{\f18}_{\AB}\|\p_2\nabla u\|^{\f18}_{\AB},
~\text{ when }\al\in(\f18,\f14];\\
&\|u\|_{\LB}^{\f12-\al}
\|\p_1u\|_{\LB}^{\f\al2}\|\p_2u\|_{\LB}^{\f\al2}\|\nablah\nabla u\|_{\AB}^{\f12}\\
&\quad\times\|\nabla u\|_{\LB}^{\al}
\|\p_1\nabla u\|_{\AB}^{\f14-\f\al2}\|\p_2\nabla u\|_{\AB}^{\f14-\f\al2},
~\text{ when }\al\in(\f14,\f12).
\end{split}
     \right.
\end{equation*}

{\bf Step 3. Continuity argument.}
Let us denote
\begin{align*}
 T^{\ast}\eqdef \sup\Bigl\{&T>0:u\text{ exists on $[0,T)$, and for any $t<T$, there holds}\\
 &\|u\|_{\LTB}+\nu_1^{\f12} \|\p_1 u\|_{\ATB}+\nu_2^{\f12} \| \p_2 u\|_{\ATB}\leq 2\|u_0\|_{\BB},\quad\text{and}\\
 &\|\nabla u\|_{\LTB}+\nu_1^{\f12} \|\p_1\nabla u\|_{\ATB}+\nu_2^{\f12}
 \|\p_2\nabla u\|_{\ATB}\leq 2\|\nabla u_0\|_{\BB}\Bigr\}.
\end{align*}
Then for all $t<T^{\ast}$, we can deduce from \eqref{4.2} that
\begin{equation}\begin{split}\label{4.12}
\|u\|_{\LB}+\nu_1^{\f12}&\|\p_1 u\|_{\AB}+\nu_2^{\f12} \|\p_2 u\|_{\AB}\\
&\leq \| u_0\|_{\BB}\Bigl(1+Ct^{\f\al2}\nu_1^{-\f18}\nu_2^{-\f38+\f\al2}
\|u_0\|^{\f12-\alpha}_{\BB}\|\nabla u_0\|^{\alpha}_{\BB}\Bigr).
\end{split}\end{equation}

On the other hand, it is not difficult to verify that before time
$T^{\ast}$, there holds
\begin{equation}\label{4.13}
\f{F(\al)}{\|\nabla u_0\|_{\BB}}\lesssim G(\al)\eqdef
\left\{
     \begin{split}
&\nu_1^{-\f18}\nu_2^{-\f38+\f\al2}
\|u_0\|^{\f12-\al}_{\BB}\|\nabla u_0\|^{\al}_{\BB},
~\text{ when }\al\in(0,\f18];\\
&\nu_1^{-\f3{16}+\f\al2}\nu_2^{-\f5{16}}
\|u_0\|^{\f12-\al}_{\BB}\|\nabla u_0\|^{\al}_{\BB},
~\text{ when }\al\in(\f18,\f14];\\
&\nu_1^{-\f18+\f\al4}\nu_2^{-\f38+\f\al4}
\|u_0\|^{\f12-\al}_{\BB}\|\nabla u_0\|^{\al}_{\BB},
~\text{ when }\al\in(\f14,\f12),
\end{split}
     \right.
\end{equation}
which together with the assumption $\nu_2\leq\nu_1$
implies that for any $\al\in(0,\f12)$, there holds
\begin{equation}\label{4.14}
F(\al)+\nu_1^{-\f18}\nu_2^{-\f38+\f\al2}
\|u_0\|^{\f12-\al}_{\BB}\|\nabla u_0\|^{1+\al}_{\BB}
\lesssim G(\al)\cdot\|\nabla u_0\|_{\BB}.
\end{equation}
And then we can deduce from \eqref{4.11} that
\begin{equation}\begin{split}\label{4.15}
\|\nabla u\|_{\LB}+\nu_1^{\f12} \|\p_1\nabla u\|_{\AB}
+\nu_2^{\f12} \|\p_2\nabla u\|_{\AB}
\leq \|\nabla u_0\|_{\BB}\Bigl(1+Ct^{\f{\alpha}{2}}G(\al)\Bigr).
\end{split}\end{equation}

Now we have obtained estimates \eqref{4.12} and \eqref{4.15},
which hold for all $t<T^{\ast}$. In order not to contradict to
the definition of $T^\ast$, there must hold
$$C |T^\ast|^{\f\al2}\Bigl(\nu_1^{-\f18}\nu_2^{-\f38+\f\al2}
\| u_0\|^{\f12-\al}_{\BB}\|\nabla u_0\|^{\al}_{\BB}+G(\al)\Bigr)\geq1,$$
which together with \eqref{4.14} implies the lifespan $T^\star$
of $u$ at least satisfies
$$T^\star\geq T^\ast\geq CG(\al)^{-\f2\al},$$
where $G(\al)$ is given by \eqref{4.13}. This completes the proof of Theorem \ref{thm3}.

\subsection{Proof of Theorem \ref{thm4}}
We first get, by taking $L^2_\h$ inner produce of \eqref{eq4} with $u$ that
\footnote{The subscript $\h$ stands for horizontal, which means
that the norm is taken on $\R_{x_1}\times\R_{x_2}$.}
\begin{equation}\begin{split}\label{4.16}
\f{d}{dt}\|u\|^2_{L^2_\h}-\nu_3\bigl(\p^2_3u\,,\,u\bigr)_{L^2_{\h}}
&=-\big(u\cdot \nabla u\,,\,u\big)_{L^2_\h}
-\big(\nabla P\,,\,u\big)_{L^2_\h}\\
&\leq\|u\|_{L^{\infty}_\h}\|\nabla u\|_{L^2_\h}\|u\|_{L^2_\h}
-\big(\nabla P\,,\,u\big)_{L^2_\h}\\
&\leq\|u\|_{H^{s_1}_\h}^2\bigl(\|\pa_3 u\|_{L^2_\h}+\|u\|_{H^{s_1}_\h}\bigr)
-\big(\nabla P\,,\,u\big)_{L^2_\h},
\end{split}\end{equation}
where in the last step, we have used the fact that
$H^{s_1}_{\h}\hookrightarrow L^{\infty}_{\h}$
since $s_1>2$.

Next, by taking $\dot{H}^{s_1}_{\h}$ inner product of \eqref{eq4} with $u$,
we get
\begin{equation}\label{4.17}
\f{d}{dt}\|u\|^2_{\dot{H}^{s_1}_\h}
-\nu_3\bigl(\p^2_3\Lambda^{s_1}_{\h}u\,,\,\Lambda^{s_1}_{\h}u\bigr)_{L^2_{\h}}
=-\cD-\cE-\big(\nabla P\,,\,u\big)_{\dH^{s_1}_\h},
\end{equation}
where the operator $\Lam_\h\eqdef(-\D_\h)^{\f12}$, and
$$\cD=\bigl(\Lambda^{s_1-1}_{\h}\nablah(u^{\h}\cdot \nabla_{\h}u)\,,\,
\Lambda^{s_1-1}_{\h}\nablah u\bigr)_{L^2_\h},\quad
\cE=\big(\Lambda^{s_1-1}_{\h}\nablah(u^{3}\cdot \p_3u)\,,\,
\Lambda^{s_1-1}_{\h}\nablah u\big)_{L^2_\h}.$$

For the first term $\cD$, let us write it as
\begin{align*}
\cD&=\big(\Lambda^{s_1-1}_{\h}(\nabla_{\h}u^{\h}\cdot \nabla_{\h}u)\,,\,
\Lambda^{s_1-1}_{\h}\nabla_{\h}u\big)_{L^2_\h}
+\big(\Lambda^{s_1-1}_{\h}(u^{\h}\cdot\nablah\nablah u)\,,\,
\Lambda^{s_1-1}_{\h}\nabla_{\h}u\big)_{L^2_\h}\\
&=\big(\Lambda^{s_1-2}_{\h}\nablah(\nabla_{\h}u^{\h}\cdot \nabla_{\h}u)\,,\,
\Lambda^{s_1-2}_{\h}\nablah^2u\big)_{L^2_\h}
+\big([\Lambda^{s_1-1}_{\h},u^{\h}\cdot\nabla_{\h}] \nabla_{\h}u\,,\,
\Lambda^{s_1-1}_{\h}\nabla_{\h}u\big)_{L^2_\h}\\
&\quad+\big(u^{\h}\cdot \nabla_{\h}\Lambda^{s_1-1}_{\h}\nabla_{\h}u\,,\,
\Lambda^{s_1-1}_{\h}\nabla_{\h}u\big)_{L^2_\h}
\eqdef\cD_1+\cD_2+\cD_3.
\end{align*}

Notice that in our aim estimate \eqref{thm4.1}, the norm
$H^{s_1,0}$ is inhomogeneous in horizontal variables.
So without loss of generality, we assume $s_1\in(2,3)$ from now on for simplification.
Then we can use Lemma \ref{product law} to get
\begin{equation}\label{4.18}
     \begin{aligned}
     |\cD_1|&\leq \|\nabla_{\h}u\otimes \nablah^2u\|_{\dH^{s_1-2}_{\h}}
     \|u\|_{\dH^{s_1}_{\h}}\\
     &\lesssim \|\nablah u\|_{\dH^{1-\f{s_1-2}2}_{\h}}^{\f12}
     \|\nablah u\|_{\dH^{1+\f{s_1-2}2}_{\h}}^{\f12}
     \|\nabla^2_{\h}u\|_{\dH^{s_1-2}_{\h}}\|u\|_{\dH^{s_1}_{\h}}\\
     &\lesssim \|u\|_{H^{s_1}_{\h}}^3.
 \end{aligned}
\end{equation}
And since $s_1-1>1$, the commutator term $\cD_2$
can be handled by using Lemma \ref{communitor} as
\begin{equation}\begin{split}
|\cD_2|\leq \|[\Lambda^{s_1-1}_{\h},u^{\h}\cdot\nabla_{\h}]\nabla_{\h}u\|_{L^2_{\h}}
\|\Lam^{s_1-1}\nabla_{\h}u\|_{L^2_{\h}}
\lesssim\|\nablah u\|_{H^{s_1-1}_{\h}}^3\lesssim\|u\|_{H^{s_1}_{\h}}^3.
\end{split}\end{equation}
And thanks to $H^{s_1-1}_{\h}\hookrightarrow L^{\infty}_{\h}$
for $s_1>2$, we get by using integrating by parts that
\begin{equation}\begin{split}
|\cD_3|&=\f12\Bigl|\int_{\R_{\h}}
\diveh u^{\h}\cdot|\Lambda^{s_1-1}_{\h}\nabla_{\h}u|^2\,dx_{\h}\Bigr|\\
&\lesssim \|\diveh u^{\h}\|_{L^{\infty}_\h}
\|\Lambda^{s_1-1}_{\h}\nabla_{\h}u\|_{L^2_\h}^2\\
&\lesssim \|u\|_{H^{s_1}_{\h}}^3.
\end{split}\end{equation}

Similarly, we can write
\begin{align*}
\cE&=\big(\Lambda^{s_1-1}_{\h}(\nabla_{\h}u^3\pa_3u)\,,\,
\Lambda^{s_1-1}_{\h}\nabla_{\h}u\big)_{L^2_\h}
+\big(\Lambda^{s_1-1}_{\h}(u^3\pa_3\nablah u)\,,\,
\Lambda^{s_1-1}_{\h}\nabla_{\h}u\big)_{L^2_\h}\\
&=\big(\Lambda^{s_1-2}_{\h}\nablah(\nabla_{\h}u^3\pa_3u)\,,\,
\Lambda^{s_1-2}_{\h}\nablah^2u\big)_{L^2_\h}
+\big([\Lambda^{s_1-1}_{\h},u^3\nabla_{\h}] \pa_3u\,,\,
\Lambda^{s_1-1}_{\h}\nabla_{\h}u\big)_{L^2_\h}\\
&\quad+\big(u^3\pa_3\Lambda^{s_1-1}_{\h}\nabla_{\h}u\,,\,
\Lambda^{s_1-1}_{\h}\nabla_{\h}u\big)_{L^2_\h}
\eqdef\cE_1+\cE_2+\cE_3.
\end{align*}
Notice that in our choice $s_1\in (2,3)$, thus we can use Lemma \ref{product law}
to deduce
\begin{equation}\begin{split}
|\cE_1|&\lesssim\Bigl(\|\nabla^2_{\h}u^{3}\cdot \p_3u\|_{\dH^{s_1-2}_{\h}}
+\|\nabla_{\h}u^{3}\cdot \p_3\nabla_{\h}u\|_{\dH^{s_1-2}_{\h}}\Bigr)
\|u\|_{\dH^{s_1}_{\h}}\\
&\lesssim\Bigl(\|\nabla^2_{\h}u^3\|_{\dH^{s_1-2}_{\h}}
\|\pa_3u\|_{\dH^{1-\f{s_1-2}2}_{\h}}^{\f12}
\|\pa_3u\|_{\dH^{1+\f{s_1-2}2}_{\h}}^{\f12}\\
&\qquad+\|\pa_3\nablah u\|_{\dH^{s_1-2}_{\h}}
\|\nabla_{\h}u^3\|_{\dH^{1-\f{s_1-2}2}_{\h}}^{\f12}
\|\nabla_{\h}u^3\|_{\dH^{1+\f{s_1-2}2}_{\h}}^{\f12}\Bigr)
\|u\|_{\dH^{s_1}_{\h}}\\
&\lesssim \|u\|^2_{H^{s_1}_{\h}}\|\pa_3u\|_{H^{s_1-1}_{\h}}.
\end{split}\end{equation}
And since $s_1-1>1$, the commutator term $\cE_2$
can be handled by using Lemma \ref{communitor} as
\begin{equation}\begin{split}\label{4.22}
|\cE_2|&\leq \|[\Lambda^{s_1-1}_{\h},u^3\nabla_{\h}]\pa_3u\|_{L^2_{\h}}
\|\Lam^{s_1-1}\nabla_{\h}u\|_{L^2_{\h}}\\
&\lesssim\|\nablah u^3\|_{H^{s_1-1}_{\h}}\|\pa_3u\|_{H^{s_1-1}_{\h}}
\|\Lam^{s_1-1}\nabla_{\h}u\|_{L^2_{\h}}\\
&\lesssim\|u\|_{H^{s_1}_{\h}}^2\|\pa_3u\|_{H^{s_1-1}_{\h}}.
\end{split}\end{equation}

Now by substituting \eqref{4.18}-\eqref{4.22} into \eqref{4.17}, we achieve
\begin{align*}
\f{d}{dt}\|u\|^2_{\dot{H}^{s_1}_\h}
& -\nu_3\bigl(\p^2_3\Lambda^{s_1}_{\h} u\,,\,
\Lambda^{s_1}_{\h} u\bigr)_{L^2_{\h}}\\
&\lesssim \|u\|^3_{H^{s_1}_{\h}}
+\|u\|^2_{H^{s_1}_{\h}}\|\p_3u\|_{H^{s_1-1}_{\h}}
-\big(u^3\pa_3\Lambda^{s_1-1}_{\h}\nabla_{\h}u\,,\,
\Lambda^{s_1-1}_{\h}\nabla_{\h}u\big)_{L^2_\h}
-\big(\nabla P\,,\,u\big)_{\dH^{s_1}_\h},
\end{align*}
which together with \eqref{4.16} leads to
\begin{align*}
\f{d}{dt}\|u\|^2_{{H}^{s_1}_\h}
& -\nu_3\bigl(\p^2_3 u\,,\, u\bigr)_{L^2_{\h}}
-\nu_3\bigl(\p^2_3\Lambda^{s_1}_{\h} u\,,\,
\Lambda^{s_1}_{\h} u\bigr)_{L^2_{\h}}\\
&\lesssim\|u\|^3_{H^{s_1}_{\h}}
+\|u\|^2_{H^{s_1}_{\h}}\|\p_3u\|_{H^{s_1-1}_{\h}}
-\big(u^3\pa_3\Lambda^{s_1-1}_{\h}\nabla_{\h}u\,,\,
\Lambda^{s_1-1}_{\h}\nabla_{\h}u\big)_{L^2_\h}
-\big(\nabla P\,,\,u\big)_{H^{s_1}_\h}.
\end{align*}
Integrating in time and the space variable $x_3$ gives
\begin{equation}\begin{split}\label{4.23}
\|u&\|_{L^\infty_t(H^{s_1,0})}^2
+\nu_3\|\pa_3 u\|_{L^2_t(H^{s_1,0})}^2
\lesssim\|u_0\|_{H^{s_1,0}}^2
+\int_0^t\Bigl(\|u\|_{L^4_\v(H^{s_1}_{\h})}^2\|u\|_{H^{s_1,0}}\\
&+\|u\|^2_{L^4_\v(H^{s_1}_{\h})}\|\pa_3u\|_{H^{s_1-1,0}}
-\big(u^3\pa_3\Lambda^{s_1-1}_{\h}\nabla_{\h}u\,,\,
\Lambda^{s_1-1}_{\h}\nabla_{\h}u\big)_{L^2}
-\big(\nabla P\,,\,u\big)_{L^2_\v(H^{s_1}_\h)}\Bigr)\,dt',
\end{split}\end{equation}
where the norm $H^{s_1,0}=L^2_\v(H^{s_1}_\h)$ can be given as follows
$$\|a\|_{H^{s_1,0}}^2=\int_{\R^3}(1+|\xi_\h|)^{2s_1}|\wh a(\xi)|^2\,d\xi.$$
By using integration by parts and the fact that $\dive u=0$, we have
$$-\big(\nabla P\,,\,u\big)_{L^2_\v(H^{s_1}_\h)}
=\big(P\,,\,\dive u\big)_{L^2_\v(H^{s_1}_\h)}=0,$$
and
\begin{align*}
-\big(u^3\pa_3\Lambda^{s_1-1}_{\h}\nabla_{\h}u\,,\,
\Lambda^{s_1-1}_{\h}\nabla_{\h}u\big)_{L^2}
&=\f12\int_{\R^3}\pa_3 u^3|\Lambda^{s_1-1}_{\h}\nabla_{\h}u|^2\,dx\\
&\lesssim\|\pa_3 u^3\|_{L^2_\v(L^\infty_\h)}
\|\Lambda^{s_1-1}_{\h}\nabla_{\h}u\|_{L^4_\v(L^2_\h)}^2\\
&\lesssim\|\diveh \uh\|_{L^2_\v(H^{s_1-1}_{\h})}
\|u\|^2_{L^4_\v(H^{s_1}_{\h})}
\end{align*}
By substituting the above two estimates into \eqref{4.23}, we achieve
\begin{equation}\begin{split}\label{4.24}
\|& u\|_{L^\infty_t(H^{s_1,0})}^2
+\nu_3\|\pa_3 u\|_{L^2_t(H^{s_1,0})}^2\\
&\lesssim\| u_0\|_{H^{s_1,0}}^2
+\int_0^t\Bigl(\| u\|_{L^4_\v(H^{s_1}_{\h})}^2\| u\|_{H^{s_1,0}}
+\|u\|^2_{L^4_\v(H^{s_1}_{\h})}\|\pa_3u\|_{H^{s_1-1,0}}\Bigr)\,dt'\\
&\lesssim\|u_0\|_{H^{s_1,0}}^2
+t^{\f34}\|\pa_3 u\|_{L^2_t(H^{s_1,0})}^{\f12}
\|u\|_{L^\infty_t(H^{s_1,0})}^{\f52}
+t^{\f14}\|\pa_3 u\|_{L^2_t(H^{s_1,0})}^{\f32}
\|u\|_{L^\infty_t(H^{s_1,0})}^{\f32}.
\end{split}\end{equation}

In the following, we shall use the continuity argument once again.
Let us denote
\begin{align*}
 T^\ast\eqdef \sup\Bigl\{T>0:~ &u\text{ exists on $[0,T)$, and for any $t<T$, there holds}\\
 &\|u\|_{L^\infty_t(H^{s_1,0})}^2
+\nu_3\|\pa_3 u\|_{L^2_t(H^{s_1,0})}^2
 <2\|u_0\|_{H^{s_1,0}}^2\Bigr\}.
\end{align*}
Similar to the local well-posedness theory of the Euler equations,
we can prove this $T^\ast$ is well-defined and positive.
Moreover, this solution $u$ will not form singularities on
$[0,T^\ast]$.

Then for any $t<T^\ast$, we can deduce from \eqref{4.24} that
$$\|u\|_{L^\infty_t(H^{s_1,0})}^2
+\nu_3\|\pa_3 u\|_{L^2_t(H^{s_1,0})}^2
\leq \|u_0\|^2_{H^{s_1,0}_{\h}}
+C\bigl(\nu^{-\f14}_3t^{\f34}+\nu^{-\f34}_3t^{\f14}\bigr)
\|u_0\|_{H^{s_1,0}}^{3}.$$
Thus in order not to contradict to
the definition of $T^\ast$, there must hold
$$C\bigl(\nu^{-\f14}_3t^{\f34}+\nu^{-\f34}_3t^{\f14}\bigr)
\|u_0\|_{H^{s_1,0}}\geq1,$$
which implies
$$T^\ast
\geq C\min \Big\{\f{\nu_3^{\f13}}{\|u_0\|^{\f43}_{H^{s_1,0}}}\,,\,
\f{\nu_3^3}{\|u_0\|_{H^{s_1,0}}^4}\Big\}.$$
Moreover, for any time $t\leq T^\ast$, there holds
\begin{equation}\label{4.25}
\|u\|_{L^\infty_t(H^{s_1,0})}^2
+\nu_3\|\pa_3 u\|_{L^2_t(H^{s_1,0})}^2
<2\|u_0\|_{H^{s_1,0}}^2.
\end{equation}

The only remaining thing to prove is the uniqueness.
For any two solutions $u_1,\,u_2$ of \eqref{eq4} with same initial data,
their difference $\du= u_1-u_2$ satisfies
\begin{equation*}
     \left\{
     \begin{array}{l}
     \pa_t\du-\nu_3\pa_3^2\du
     +u\cdot\nabla\du+\du\cdot\nabla u_2+\nabla(P_1-P_2)=0,\\
     \dive \du= 0,\\ \du|_{t=0} =0.
     \end{array}
     \right.
\end{equation*}
The $L^2$ energy estimate gives
\begin{equation}\begin{split}\label{4.26}
\f12\f{d}{dt}\|\du\|_{L^2}^2+\nu_3\|\pa_3\du\|_{L^2}^2
&\leq-\int_{\R^3}(\du\cdot\nabla u_2)\du\,dx\\
&=-\int_{\R^3}(\du^\h\cdot\nablah u_2)\du\,dx
+\int_{\R^3}u_2\cdot(\du^3\pa_3\du+\du\pa_3\du^3)\,dx\\
&\leq\|\du\|_{L^2}^2\|\nablah u_2\|_{L^\infty}
+\f12\nu_3\|\pa_3\du\|_{L^2}^2
+\nu_3^{-1}\|\du\|_{L^2}^2\|u_2\|_{L^\infty}^2
\end{split}\end{equation}
On the other hand, in view of \eqref{4.25},
for any $t\leq T^\ast$, there holds
\begin{equation}\begin{split}\label{4.27}
\int_0^t\bigl(\|\nablah u_2(t')\|_{L^\infty}
+\nu_3^{-1}\|u_2(t')\|_{L^\infty}^2\bigr)\,dt'
\lesssim& t^{\f34}\|u_2\|_{L^\infty_t(H^{s_1,0})}^{\f12}
\|\pa_3 u\|_{L^2_t(H^{s_1,0})}^{\f12}\\
&+\nu_3^{-1}t^{\f12}\|u\|_{L^\infty_t(H^{s_1,0})}
\|\pa_3 u\|_{L^2_t(H^{s_1,0})}<\infty.
\end{split}\end{equation}
Thanks to \eqref{4.27} and the fact that $\du|_{t=0}=0$,
we can deduce by applying Gronwall's inequality to \eqref{4.26}
that $\du$ indeed vanishes on $[0,T^\ast]$.
This completes the proof of Theorem \ref{thm3}.

\setcounter{equation}{0}
\section{GWP with viscous coefficients large in one direction}\label{Sect5}

In this section, we shall prove that both \eqref{1.1} and \eqref{ANS}
would admit a global strong solution
provided one of their viscous coefficients is sufficiently large.

\subsection{The proof of Corollary \ref{col1.1}}
Since initially $u_0\in L^2$, the energy inequality gives
\begin{equation}\label{5.1}
\|u\|_{L^\infty_t(L^2)}^2+2\sum_{i=1}^3\nu_i\|\pa_i u\|_{L^2_t(L^2)}^2
\leq\|u_0\|_{L^2}^2,\quad\forall\ t>0.
\end{equation}
By Theorem \ref{thm1}, we know the corresponding strong solution
at least exists to
$$T^{\star}\geq C\nu^{\frac{p}{p-3}}_2(\nu_1\nu_2\nu_3)^{\frac{1}{p-3}}
\|u_0\|_{L^p}^{-\frac{2p}{p-3}},$$
which together the condition \eqref{condinu11} on $\nu_1$ ensures
$$T^\star\gg\nu_3^{-5}\|u_0\|_{L^2}^4.$$
Then by applying Chebyshev's inequality on \eqref{5.1},
and in view of the assumption that $0<\nu_3\leq\nu_2\leq\nu_1$,
we know there must exist some $t_0\in(0,T^\star)$ such that
$$2\nu_3\|\nabla u(t_0)\|_{L^2}^2
\leq2\sum_{i=1}^3\nu_i\|\pa_i u(t_0)\|_{L^2}^2
\leq\f{1}{T^\star}\|u_0\|_{L^2}^2,$$
and thus
$$\|u(t_0)\|_{L^2}\|\nabla u(t_0)\|_{L^2}
\leq\|u_0\|_{L^2}\cdot\f{1}{\sqrt{2\nu_3T^\star}}\|u_0\|_{L^2}
\ll\nu_3^2.$$
This in particular means that from time $t_0<T^\star$, this strong solution
$u$ has already become to a small solution. Then due to the classical
well-posedness result for N-S with small initial data
(see for instance \cite{Leray}), this strong solution $u$
indeed exists globally in time.

\subsection{The proof of Corollary \ref{col1.2}}
Exactly along the same line to the proof of \eqref{4.2},
but by using Lemma \ref{lem2.5} with $\al=0$ instead of
$\al\in(0,\f12)$ there, we can achieve
\begin{equation}\begin{split}\label{5.2}
\|u\|_{\LB}&+\nu_1^{\f12} \|\p_1 u\|_{\AB}+\nu_2^{\f12}\|\p_2 u\|_{\AB}\\
&\leq \|u_0\|_{\BB}+C\|u\|^{\f12}_{\LB}
\|\p_1u\|^{\f14}_{\AB}\|\p_2u\|^{\f14}_{\AB}
\|\nabla_{\h}u\|^{\f12}_{\AB}.
\end{split}\end{equation}
Then similar to the proof of Theorem \ref{thm3} in Subsection \ref{sub4.1},
let us denote
\begin{align*}
 T^{\ast}\eqdef \sup\Bigl\{&T>0:u\text{ exists on $[0,T)$, and for any $t<T$, there holds}\\
 &\|u\|_{\LTB}+\nu_1^{\f12} \|\p_1 u\|_{\ATB}+\nu_2^{\f12} \| \p_2 u\|_{\ATB}\leq 2\|u_0\|_{\BB}\Bigr\}.
\end{align*}
The local well-posedness result in \cite{Pa02} guarantees that
this $T^\ast$ is well-defined and positive.

We shall prove by contradiction. If $T^\ast<\infty$,
then for all $t\leq T^{\ast}$, we can deduce from \eqref{5.2} that
$$\|u\|_{\LB}+\nu_1^{\f12}\|\p_1 u\|_{\AB}+\nu_2^{\f12} \|\p_2 u\|_{\AB}
\leq \| u_0\|_{\BB}\Bigl(1+C\nu_1^{-\f18}\nu_2^{-\f38}
\|u_0\|^{\f12}_{\BB}\Bigr).$$
As a result, as long as $\nu_1$ is sufficiently large such that
$$C\nu_1^{-\f18}\nu_2^{-\f38}
\|u_0\|^{\f12}_{\BB}<\f12,$$
we can deduce
$$\|u\|_{\LB}+\nu_1^{\f12}\|\p_1 u\|_{\AB}+\nu_2^{\f12} \|\p_2 u\|_{\AB}
\leq \f32\|u_0\|_{\BB},\quad\forall\ t\leq T^{\ast}.$$
This contradicts to the choice of $T^\ast$.
Thus there must be $T^\ast=\infty$, and the following estimate
holds uniformly in time:
$$\|u\|_{\LB}+\nu_1^{\f12}\|\p_1 u\|_{\AB}+\nu_2^{\f12} \|\p_2 u\|_{\AB}
\leq 2\| u_0\|_{\BB},\quad\forall\ t>0.$$
This completes the proof of this corollary.

\medskip

\noindent {\bf Acknowledgments.}
The authors would like to thank Prof. Ping Zhang for his careful reading
and some valuable comments.
Liu is supported by NSF of China under grant 12101053,
and the Fundamental Research Funds
for the Central Universities under Grant 310421118.

\end{document}